\documentclass[12pt]{article}
\usepackage{amsfonts}
\usepackage{amsmath, amssymb, amsthm}
\usepackage{enumerate}
\usepackage{epsfig}
\usepackage{verbatim}
\usepackage{exscale}

\numberwithin{equation}{section}

\newtheorem{thm}{Theorem}[section]
\newtheorem{lem}[thm]{Lemma}
\newtheorem{proposition}[thm]{Proposition}

\newtheorem{cor}[thm]{Corollary}
\newtheorem{defn}[thm]{Definition}

\providecommand{\R}{\ensuremath{\mathbb{R}}}
\providecommand{\N}{\ensuremath{\mathbb{N}}}

\providecommand{\Sref}{\ensuremath{S_{ST}}}
\providecommand{\Iref}{\ensuremath{I_{ST}}}

\newenvironment{numberedproof}[1]{{\bf Proof of #1:}}{{}\hfill{\hbox{$\Box$}}\par\bigskip}

\newif\iftechreport
\techreportfalse
\iftechreport
\else
\fi

\begin{document}
\iftechreport
\title{Robust exponential convergence of $hp$-FEM in balanced norms for singularly perturbed
reaction-diffusion equations (extended version)}
\else
\title{Robust exponential convergence of $hp$-FEM in balanced norms for singularly perturbed
reaction-diffusion equations}
\fi
\author{J. M. Melenk \\
Institut f\"{u}r Analysis und Scientific Computing\\
Vienna University of Technology\\
Wiedner Hauptstrasse 8-10, A-1040 Wien \\
AUSTRIA \vspace{0.125cm} \\
and\\
C. Xenophontos\\
Department of Mathematics and Statistics\\
University of Cyprus \\
P.O. BOX 20537\\
Nicosia 1678\\
CYPRUS}
\maketitle

\begin{abstract}
The $hp$-version of the finite element method is applied to a singularly perturbed
reaction-diffusion equation posed on an interval or a two-dimensional domain with 
an analytic boundary. On suitably designed \emph{Spectral Boundary Layer meshes}, 
robust exponential convergence in a ``balanced'' norm is shown. 
This ``balanced'' norm is an $\varepsilon$-weighted $H^1$-norm, where 
the weighting in terms of the singular perturbation parameter $\varepsilon$ is such that, 
in contrast to the standard energy norm, boundary layer contributions do {\em not}
vanish in the limit $\varepsilon \rightarrow 0$. 
Robust exponential convergence in the maximum norm is also established. We illustrate 
the theoretical findings with two numerical experiments.
\end{abstract}

\section{Introduction}
\label{sec:intro}

\label{intro} The numerical solution of singularly perturbed problems has
been studied extensively over the last decades (see, e.g., the books 
\cite{mos,rst} and the references therein). These problems typically feature 
boundary layers (and, more generally, also internal layers). Their  
resolution requires the use of strongly refined, layer-adapted meshes. 
In the context of fixed order methods, well-known representatives of such 
meshes include the Bakhvalov mesh \cite{B} and the Shishkin mesh \cite{Shishkin2}. 
For the $p$/$hp$-version Finite Element Method (FEM) or for spectral methods, 
the \emph{Spectral Boundary Layer mesh} \cite{schwab-suri96,melenk97,mB} 
is essentially the smallest mesh that permits the resolution of boundary layers
(see Definition~\ref{SBL} ahead for the 1D version and Section~\ref{sec:2D-meshes} 
for a realization in 2D). 

The use of the above mentioned meshes can lead to robust convergence, i.e., 
convergence uniform in the singular perturbation parameter. For the reaction-diffusion
equations (\ref{eq:1D-problem}), (\ref{eq:2D-problem}) under consideration here, 
the FEM is naturally analyzed in the \emph{energy norm} 
(\ref{eq:energy-norm-1D}), (\ref{eq:energy-norm-2D}), which 
is simply the norm induced by the inner-product defined by the bilinear form of the 
variational problem;  robust convergence of the $h$-FEM 
on Shishkin meshes can be found, for example, in \cite{rst} and robust exponential
convergence on \emph{Spectral Boundary Layer meshes} is shown in \cite{melenk97,mB}. The 
(natural) energy norm associated with this boundary value problem is rather weak 
in that the layer contributions are not ``seen'' by the energy norm; that is, 
the energy norm of the layer contribution vanishes as the 
singular perturbation parameter $\varepsilon$ tends to zero whereas the energy norm of the 
smooth part of the solution does not.  This has sparked the recent work 
\cite{lin-stynes11,roos-franz11,roos-schopf11} to study the convergence of the $h$-FEM
in norms stronger than the energy norm. The analysis of 
\cite{lin-stynes11,roos-franz11,roos-schopf11} is performed in an 
$\varepsilon$-weighted $H^1$-norm 
which is \emph{balanced} in the sense that both the smooth
part and the layer part are (generically) bounded away from zero uniformly in $\varepsilon$; 
both energy norm (see (\ref{eq:energy-norm-1D}), (\ref{eq:energy-norm-2D}) for the 1D and 2D case, respectively) 
and balanced norm (see (\ref{eq:balanced-norm-1D}), (\ref{eq:balanced-norm-2D})) 
are $\varepsilon$-weighted $H^1$-norms but they differ in the $\varepsilon$-scaling. 
Robust convergence in this balanced norm is shown in \cite{lin-stynes11,roos-franz11,roos-schopf11} 
if Shishkin meshes are employed. 
We show in the present work that this analysis can be extended to 
the $hp$-version FEM on \emph{Spectral Boundary Layer meshes} to give robust exponential convergence
of the $hp$-version FEM in this balanced norm. An additional outcome of our convergence 
analysis in the balanced norm is the robust exponential convergence in the maximum norm. 

It is worth mentioning that robust exponential convergence of the $hp$-FEM on \emph{Spectral Boundary Layer
meshes} in the balanced norm was shown earlier in special cases. For example, for the case of 
equations with constant coefficients and polynomial right-hand sides, \cite{schwab-suri96} observes 
that the smooth part of the asymptotic expansion is again polynomial and therefore in the finite element 
space. It follows that a factor $\varepsilon^{1/2}$ is gained in the convergence estimate and leads to robust 
exponential convergence in the balanced norm. A more detailed discussion of similar effects can be found 
in the concluding remarks of \cite{melenk-xenophontos-oberbroeckling13a} and in the section
with numerical results in \cite{melenk-xenophontos-oberbroeckling13b}.

Let us briefly discuss the ideas underlying our analysis. 
Asymptotic expansions may be viewed as a tool to decompose the solution into
components associated with different length scales. Roughly speaking, our 
analysis in balanced norms mimicks this technique on the discrete level in that 
the Galerkin approximation is likewise decomposed into components associated 
with different length scales. In total, our analysis involves the following ideas: 

\begin{enumerate}
\item An analysis of the difference between the FEM approximation and a
Galerkin approximation to a \emph{reduced problem}.

\item A stable decomposition of the FEM space on the layer-adapted mesh into
fine and coarse components. This decomposition relies essentially on
strengthened Cauchy-Schwarz inequalities. 
\end{enumerate}

Throughout the paper we will utilize the usual Sobolev space notation 
$H^{k}\left( \Omega \right)$ 
to denote the space of functions on $\Omega$ with 
weak derivatives up to order $k$ in $L^{2}\left( \Omega\right)$, equipped with the
norm $\left\Vert \cdot\right\Vert_{k,\Omega}$ and seminorm 
$\left\vert \cdot \right\vert_{k,\Omega}$. We will also use the space 
$H_{0}^{1}\left( \Omega \right) =\left\{ u\in H^{1}\left(
\Omega \right) :\left. u\right\vert _{\partial \Omega}=0\right\}$, 
where $\partial \Omega$ denotes the boundary of $\Omega$. The norm of the
space $L^\infty(\Omega)$ of essentially bounded functions is denoted by 
$\|\cdot\|_{\infty,\Omega}$. The letters $C$, $c$ will be used to denote generic
positive constants, independent of any discretization or singular
perturbation parameters and possibly having different values in each
occurrence. Finally, the notation $A\lesssim B$ means the existence of a
positive constant $C$, which is independent of the quantities $A$ and $B$
under consideration and of the singular perturbation parameter $\varepsilon$, 
such that $A\leq CB$. 

\section{The one-dimensional case}
\label{sec:1D}
We start with the one-dimensional case as many of the ideas can be seen 
in this setting already. 
\subsection{Problem formulation and solution regularity}
We consider the following model problem: Find $u$ such that 
\begin{subequations}
\label{eq:1D-problem}
\begin{align}
\label{de}
-\varepsilon ^{2}u^{\prime \prime }+bu&=f\text{ in }I=(0,1),\\
u(0)&=u(1)=0.  \label{bc}
\end{align}
\end{subequations}
The parameter $0<\varepsilon \leq 1$ is given, as are the functions $b>0$ and $f$,
which are assumed to be analytic on $\overline{I}=[0,1]$. In particular, we assume 
that there exist constants $C_{f}$, $\gamma_{f}$, $C_{b}$, $\gamma _{b}$, $c_b >0$, such that 
\begin{equation}
\left\{ 
\begin{array}{c}
\left\Vert f^{(n)}\right\Vert _{\infty ,I}\leq C_{f}\gamma _{f}^{n}n!\quad
\forall \;n\in \mathbb{N}_{0}, \\ 
\left\Vert b^{(n)}\right\Vert _{\infty ,I}\leq C_{b}\gamma _{b}^{n}n!\quad
\forall \;n\in \mathbb{N}_{0}, \\
b(x) \ge c_b  > 0 \qquad \forall x \in \overline{I}. 
\end{array}
\right.   \label{analytic_data}
\end{equation}
The variational formulation of (\ref{eq:1D-problem}) reads: Find 
$u\in H_{0}^{1}\left( I\right)$ such that 
\begin{equation}
{\mathcal{B}}_{\varepsilon}\left( u,v\right) ={\mathcal{F}}\left( v\right) \;\;\forall
\;v\in H_{0}^{1}\left( I\right),   \label{BuvFv}
\end{equation}
where, with $\left\langle \cdot ,\cdot \right\rangle _{I}$ the usual $L^{2}(I)$ inner product, 
\begin{eqnarray}
{\mathcal{B}}_{\varepsilon }\left( u,v\right) &=&\varepsilon
^{2}\left\langle u^{\prime },v^{\prime }\right\rangle _{I}+\left\langle
bu,v\right\rangle _{I},  \label{Buv} \\
{\mathcal{F}}\left( v\right) &=&\left\langle f,v\right\rangle _{I}.
\label{Fv}
\end{eqnarray}
The bilinear form ${\mathcal{B}}_{\varepsilon}\left( \cdot ,\cdot \right)$
given by (\ref{Buv}) is coercive with respect to the \emph{energy norm} 
\begin{equation}
\label{eq:energy-norm-1D}
\left\Vert u\right\Vert _{E,I}^{2}:={\mathcal{B}}_{\varepsilon}\left( u,u\right) ,
\end{equation}
i.e., 
\begin{equation*}
{\mathcal{B}}_{\varepsilon }\left( u,u\right) \geq \left\Vert u\right\Vert
_{E,I}^{2}\;\;\forall \;u\in H_{0}^{1}\left( I\right) .
\end{equation*}
The solution $u$ is analytic in ${I}$ and features boundary layers at the endpoints. 
Its regularity was described in \cite{melenk97} (our presentation below follows
\cite[Prop.~{2.2.1}]{mB}) both in terms of classical differentiability
(see Proposition~\ref{prop:1D-regularity}, (\ref{item:prop:1D-regularity-i})) as well as asymptotic expansions 
(see Proposition~\ref{prop:1D-regularity}, (\ref{item:prop:1D-regularity-ii})): 
\begin{proposition}[{\cite[Prop.~{2.2.1}]{mB}, \cite{melenk97}}] 
\label{prop:1D-regularity}
Assume (\ref{analytic_data}) and let $u \in H^1_0(I)$ be the solution of (\ref{eq:1D-problem})
Then:  
\begin{enumerate}[(i)]
\item 
\label{item:prop:1D-regularity-i}
There are constants $C$, $K > 0$ independent of $\varepsilon \in (0,1]$ such that 
$\|u^{(n)}\|_{L^2(I)} \leq C K^n \max\{n+1,\varepsilon^{-1}\}^n$ for all $n \in \N_0$.
\item 
\label{item:prop:1D-regularity-ii}
$u$ can be decomposed as 
$\displaystyle u=w+u^{BL}+r
$
where, 
for some constants $C_{w}$, $\gamma _{w}$, $C_{BL}$, $\gamma _{BL}$, $C_{r}$, 
$\gamma _{r}$, $b >0$ independent of $\varepsilon \in (0,1]$, 
\begin{subequations}
\begin{eqnarray}  \label{eq:regularity-by-decomposition-scalar}
\left\Vert w^{(n)}\right\Vert _{\infty ,I}&\leq& C_{w}\gamma _{w}^{n}n^{n} \quad \forall n \in {\mathbb{N}}_0, 
\label{smooth} \\
\left\vert \left( u^{BL}\right) ^{(n)}(x)\right\vert &\leq& C_{BL}\gamma
_{BL}^{n} \max\{n+1,\varepsilon^{-1}\}^{n} e^{-b \operatorname*{dist}(x,\partial I)/\varepsilon } \quad  \forall n \in {\mathbb{N}}_0, 
\label{BL} \\
\Vert r^{(n)}\Vert _{0,I} &\leq& C_{r}\varepsilon ^{2-n}e^{-\gamma
_{r}/\varepsilon }, \; n\in \{0,1,2\}.   \label{remainder}
\end{eqnarray}
\end{subequations}
\end{enumerate}
\end{proposition}
\subsection{High order FEM} 
The discrete version of the variational formulation (\ref{BuvFv}) reads: Given 
$V_N \subset H^1_0(\Omega)$ 
find $u_{FEM} \in V_N$ such that 
\begin{equation} 
\label{eq:abstract-FEM}
{\mathcal{B}}_{\varepsilon}\left( u_{FEM},v\right) ={\mathcal{F}}\left( v\right) \quad \forall v \in V_N. 
\end{equation}
In order to define the FEM space $V_N$, let 
$\Delta =\left\{ 0=x_{0}<x_{1}<...<x_{N}=1\right\}$ be an arbitrary partition of 
$I=\left( 0,1\right)$ and set 
\[
I_{j}=\left[ x_{j-1},x_{j}\right] ,\quad h_{j}=x_{j}-x_{j-1},\quad j=1,...,N.
\]
Also, define the reference element $\Iref=[ -1,1]$
and note that it can be mapped onto the $j^{\text{th}}$ element $I_{j}$ by
the standard affine mapping 
$x=M_{j}(t)=\frac{1}{2}\left( 1-t\right) x_{j-1}+\frac{1}{2}\left( 1+t\right) x_{j}$.
%
With $\Pi _{p}\left( \Iref\right) $ the space of polynomials of degree $\leq p$ 
on $\Iref$ (and with $\circ $ denoting composition of functions),
we define the finite dimensional subspace as 
\begin{eqnarray*}
{\mathcal S}^{{p}}(\Delta ) &=&\left\{ v\in H^{1}\left( I\right) :v\circ
M_{j}\in \Pi _{p_{j}}(\Iref),j=1,...,N\right\} , \\
{\mathcal S}_{0}^{{p}}(\Delta ) &=&S^{{p}}(\Delta )\cap H_{0}^{1}(I).
\end{eqnarray*}
We restrict our attention here to constant polynomial degree $p$ for all
elements, i.e., $p_{j}=p$, $j=1,\ldots ,N$; clearly, more
general settings with variable polynomial degree are possible. The following 
\emph{Spectral Boundary Layer mesh} is essentially the minimal mesh that
yields robust exponential convergence.

\begin{defn}[Spectral Boundary Layer mesh]
\label{SBL}
For $\lambda>0$, $p\in \mathbb{N}$ and $0<\varepsilon \leq 1$, define the
Spectral Boundary Layer mesh $\Delta_{BL}(\lambda,p)$ as 
$$
\Delta_{BL}(\lambda,p):= 
\begin{cases}
\{0,\lambda p \varepsilon,1-\lambda p \varepsilon,1\} & \mbox{ if $\lambda p \varepsilon < 1/4$} \\
\{0,1\} & \mbox{ if $\lambda p \varepsilon \ge 1/4$}.  
\end{cases}
$$
The spaces $S(\lambda,p)$ and $S_0(\lambda,p)$ of piecewise polynomials of degree $p$ are given by 
\begin{eqnarray*}
S(\lambda,p)&:=& {\mathcal S}^p(\Delta_{BL}(\lambda,p)), \\
S_0(\lambda,p)&:=& {\mathcal S}^p_0(\Delta_{BL}(\lambda,p)) = S(\lambda,p) \cap H^1_0(I). 
\end{eqnarray*}
\end{defn}
We quote the following result from \cite{melenk97}.

\begin{proposition}[{\cite[Thm.~{16}]{melenk97}}]
\label{prop:melenk97} 
Assume that (\ref{analytic_data}) holds and let $u$ be the solution
to (\ref{BuvFv}). 
Then, there exists
$\lambda_0>0$ (depending only on $b$, $f$) such that for every $\lambda \in (0,\lambda_0)$ there  
are 
$C$, $\sigma >0$, independent of $\varepsilon \in (0,1]$  and $p\in\N$ such
that 
\begin{equation}
\label{eq:prop:melenk97-10} 
\inf_{v \in S_0(\lambda,p)} \|u - v\|_{E,I} \leq C e^{-\sigma p}. 
\end{equation}
By C\'ea's Lemma the Galerkin approximation $u_{FEM} \in S_0(\lambda,p)$ satisfies 
$\left\Vert u_{FEM}-u\right\Vert _{E,I}\sim\left\Vert
u_{FEM}-u\right\Vert _{0,I}+\varepsilon \left\Vert \left( u_{FEM}-u\right)
^{\prime }\right\Vert _{0,I}\leq Ce^{-\sigma p}.
$
%
%
\end{proposition}
Define the \emph{balanced} norm by 
\begin{equation}
\label{eq:balanced-norm-1D} 
\|v\|^2_{balanced,I}:= 
\|v\|^2_{0,I} + \varepsilon \|v^\prime\|^2_{0,I}.  
\end{equation}
We note that the balanced norm $\|\cdot\|_{balanced,I}$ is stronger 
than the energy norm $\|\cdot\|_{E,I}$ of (\ref{eq:energy-norm-1D}). 
In Lemma~\ref{lemma:1D-approximation} 
below, we will show that the approximation result (\ref{eq:prop:melenk97-10}) can be sharpened to 
$$
\inf_{v \in S_0(\lambda,p)} \|u - v\|_{balanced,I} \leq C e^{-\sigma p}. 
$$
The key step towards this result is a better treatment of the boundary layer part than it is 
done in \cite[Thm.~{16}]{melenk97}. This modification is due to \cite{schwab-suri96}. For future 
reference we formulate this modification as a separate lemma: 
\begin{lem}
\label{lemma:1D-bdy-layer-approximation}
Let $\varepsilon \in (0,1]$. Let the function $v$ satisfy on $I = (0,1)$ the estimate 
\begin{equation}
\label{eq:lemma:1D-bdy-layer-approximation-10}
|v^{(n)}(x)| \leq C_{v} \gamma^n \max\{n+1,\varepsilon^{-1}\}^n e^{-x/\varepsilon} 
\qquad \forall x \in I, \quad \forall n \in \N_0. 
\end{equation}
Then there are constants $C$, $\beta$, $\eta > 0$ (depending only on $\gamma$) such that the following is true: 
Let $\Delta$ be any mesh with a mesh point $\xi \in (0,1]$ that satisfies 
\begin{equation}
\label{eq:lemma:1D-bdy-layer-approximation-20}
\frac{\xi}{p\varepsilon} \leq \eta. 
\end{equation}
Then there exists an 
approximation $I_p v \in {\mathcal S}^p(\Delta)$ with $I_p v(0) = v(0)$
and $I_p v (1) = v(1)$ as well as the approximation properties 
\begin{align}
\nonumber 
& 
\|v - I_p v \|_{\infty,(0,\xi)} 
+\xi^{-1/2}  \|v - I_p v \|_{0,(0,\xi)} 
+\xi^{1/2}  \|v - I_p v \|_{1,(0,\xi)},  \\
\label{eq:lemma:1D-bdy-layer-approximation-100}
&\leq C C_{v}\left[ \frac{\xi}{p \varepsilon} e^{-\beta p} + e^{-\xi/\varepsilon}\right], \\
\label{eq:lemma:1D-bdy-layer-approximation-110}
& \|v - I_p v \|_{\infty,(\xi,1)} \leq C C_{v} e^{-\xi/\varepsilon}, \\
\label{eq:lemma:1D-bdy-layer-approximation-120}
& \|v - I_p v \|_{0,(\xi,1)} + 
 {\varepsilon} \|v - I_p v \|_{1,(\xi,1)} 
\leq C C_{v}\sqrt{\varepsilon} e^{-\xi/\varepsilon}.  
\end{align}
\end{lem}
\begin{proof} 
We will assume that $\xi \in (0,1/2)$; in the converse, ``asymptotic'' case 
we have $ \varepsilon^{-1} \lesssim p$ so that 
a suitable
approximation on a single element may be taken (e.g., the Gau{\ss}-Lobatto interpolant 
or the operator ${\mathcal I}_p$ discussed in detail in \cite[Thm.~{3.14}]{schwab98} 
and \cite[Sec.~{3.2.1}]{melenk-xenophontos-oberbroeckling13a}). 

It suffices to assume that the mesh consists of the two elements $I_1:= (0,\xi)$ and $I_2:=(\xi,1)$. 
We construct $I_p v$ separately on the two elements, starting with $I_1$. 

On $I_1$, we construct $I_p v$ in two steps. In the first step, we let $\pi^1 \in \Pi_p$ 
be the polynomial (on $I_1$) given by \cite[Lemma~{3.8}]{melenk-xenophontos-oberbroeckling13a}. 
It interpolates in the endpoints $0$, $\xi$ of the interval $I_1$, i.e., 
\begin{align}
\label{eq:lemma:1D-bdy-layer-approximation-10000}
& \pi^1(0) = v(0), 
\qquad 
\pi^1(\xi) = v(\xi). 
\end{align}
Furthermore, 
\cite[Lemma~{3.8}]{melenk-xenophontos-oberbroeckling13a} asserts the existence of $\eta > 0$ 
such the constraint (\ref{eq:lemma:1D-bdy-layer-approximation-20}) implies 
\begin{align}
\label{eq:lemma:1D-bdy-layer-approximation-20000}
&\xi^{-1} \|\pi^1 - v\|_{0,I_1} + |\pi^1 - v|_{1,I_1} 
\leq C C_{v} \frac{\xi^{1/2}}{p \varepsilon} e^{-\beta p}. 
\end{align}
(Note that \cite[Lemma~{3.8}]{melenk-xenophontos-oberbroeckling13a} constructs an approximation on the 
reference element $I_{ST}$ instead of $I_1$. It is applicable with $K = \varepsilon^{-1}$ 
and $h = \xi$). 
The 1D Sobolev embedding theorem in the form 
$\|v\|_{\infty,J} \lesssim |J|^{-1/2} \|v\|_{0,J} + |J|^{1/2} \|v^\prime\|_{0,J}$ (where $|J|$ denotes the length
of the interval $J$) gives 
$$
\xi^{-1/2} \|\pi^1-  v\|_{\infty,I_1} + 
\xi^{-1} \|\pi^1 - v\|_{0,I_1} + |\pi^1 - v|_{1,I_1} \leq 
C C_{v} \frac{\xi^{1/2}}{p \varepsilon} e^{-\beta p}. 
$$
In the second step, we modify $\pi^1$ as proposed in \cite{schwab-suri96} in order to obtain a better approximation on 
the element $I_2$. We define $\pi^2 \in \Pi_p$ on $I_1$ as 
$$
\pi^2(x):= \pi^1(x) - \frac{x}{\xi} (1-\sqrt{\varepsilon}) v(\xi), 
$$
so that $\pi^2(\xi) = \pi^1(\xi) - (1-\sqrt{\varepsilon}) v(\xi) = \sqrt{\varepsilon} v(\xi)$. 
In view of $|v(\xi)| \leq C_{v}  e^{-\xi/\varepsilon}$, this modification leads to 
$$
\xi^{-1/2} \|\pi^2-  v\|_{\infty,I_1} + 
\xi^{-1} \|\pi^2 - v\|_{0,I_1} + |\pi^2 - v|_{1,I_1} \leq 
C C_{v} \left[ \frac{\xi^{1/2}}{p \varepsilon} e^{-\beta p}
+ \xi^{-1/2} e^{-\xi/\varepsilon} \right]. 
$$
We take $(I_p v)|_{I_1} = \pi^2$, and this 
shows (\ref{eq:lemma:1D-bdy-layer-approximation-100}). 
On $I_2$, we take $(I_p v)|_{I_2}$ 
as the linear interpolant between the values $\pi^2(\xi) = \sqrt{\varepsilon} v(\xi)$ at $\xi$ 
and $v(1)$ at $1$. We immediately get 
\begin{equation}
\label{eq:lemma:1D-bdy-layer-approximation-1000}
\|I_p v \|_{\infty,I_2} + \|(I_p v)^\prime \|_{\infty,I_2} \leq 
C  \sqrt{\varepsilon} |v(\xi)|  
\leq C C_{v} \sqrt{\varepsilon} e^{-\xi/\varepsilon}. 
\end{equation}
Furthermore, for $v$ we have 
\begin{equation}
\label{eq:lemma:1D-bdy-layer-approximation-2000}
\|v\|_{\infty,I_2} + \varepsilon^{-1/2} \|v\|_{0,I_2} + \sqrt{\varepsilon} \|v\|_{1,I_2} 
\leq C C_{v} e^{-\xi/\varepsilon}.  
\end{equation}
(\ref{eq:lemma:1D-bdy-layer-approximation-1000}) and 
(\ref{eq:lemma:1D-bdy-layer-approximation-2000}) imply, along with the triangle inequality, then 
(\ref{eq:lemma:1D-bdy-layer-approximation-110}), 
(\ref{eq:lemma:1D-bdy-layer-approximation-120}). 
\end{proof}
Lemma~\ref{lemma:1D-bdy-layer-approximation} shows that boundary layer functions can be approximated
at a robust exponential rate in various norms including $L^\infty$ and the energy norm 
(\ref{eq:energy-norm-1D}), if the mesh is suitably chosen. We now show approximability of solutions 
to (\ref{BuvFv}) in the balanced norm (\ref{eq:balanced-norm-1D}):
\begin{lem}
\label{lemma:1D-approximation}
Assume that (\ref{analytic_data}) holds and let $u$ be the solution
to (\ref{BuvFv}). Then there are constants $\lambda_0$, $C$, $\beta > 0$ (depending only on the constants
appearing in (\ref{analytic_data})) such that for every $\lambda \in (0,\lambda_0]$, 
$\varepsilon \in (0,1]$, $p \in \N$, there 
exists an approximant $I_p u \in {\mathcal S}^p_0(\Delta_{BL}(\lambda,p))$ that satisfies 
\begin{subequations}
\label{eq:lemma:1D-approximation-10}
\begin{eqnarray}
\|u - I_p u\|_{\infty,I} &\leq& C e^{-\beta \lambda p}, \\
\|u - I_p u\|_{0,I} + \sqrt{\lambda p \varepsilon} \| (u - I_p u)^\prime\|_{0,I} &\leq & C e^{-\beta \lambda p}.
\end{eqnarray}
\end{subequations}
\end{lem}
\begin{proof}
The proof follows the lines of \cite[Thm.~{16}]{melenk97}. 
For case of $p\varepsilon$ sufficiently small, Proposition~\ref{prop:1D-regularity} decomposes the solution 
$u$ as $u = w + u^{BL} + r$. The approximation of $w$ and $r$ is done as in 
\cite[Thm.~{16}]{melenk97}. The treatment of the boundary layer part $u^{BL}$ of \cite[Thm.~{16}]{melenk97} is 
replaced with an appeal to Lemma~\ref{lemma:1D-bdy-layer-approximation}. We remark that 
slightly sharper estimates are possible if one formulates bounds for $u - I_p u$ on the two elements
$(0,\lambda p \varepsilon)$ and $(\lambda p \varepsilon,1)$ separately. 
\end{proof}
\subsection{Robust exponential convergence in a balanced norm}
The goal of this article is to improve on Proposition~\ref{prop:melenk97} by showing 
that the Galerkin error $u - u_{FEM}$ convergences at a robust exponential rate also 
in the balanced norm $\|\cdot\|_{balanced,I}$: 
\begin{thm}
\label{thm:balanced-norm-1D}
Assume (\ref{analytic_data}). Let $u$ solve (\ref{BuvFv}) and $u_{FEM} \in S_0(\lambda,p)$ 
be obtained by (\ref{eq:abstract-FEM}) based on the Spectral Boundary Layer mesh $\Delta_{BL}(\lambda ,p)$. 
Then there exists $\lambda _0 > 0$ (depending solely on $b$ and $f$) 
such that for every $\lambda \in (0,\lambda _0)$ there are constants $C$, $\sigma  > 0$ such that 
for every $\varepsilon \in (0,1]$, $p \in \N$
\begin{equation}
\label{eq:balanced_estimate-1D}
\left\Vert u-u_{FEM}\right\Vert _{0,I}+\sqrt{\varepsilon}\left\Vert \left(
u - u_{FEM}\right) ^{\prime }\right\Vert _{0,I}\leq Ce^{-\sigma p}.
\end{equation}
\end{thm}
The remainder of this section is devoted to the proof of Theorem~\ref{thm:balanced-norm-1D}. 
Before that, we note a consequence of Theorem~\ref{thm:balanced-norm-1D}: 
\begin{cor}
\label{cor:max-norm-estimate-1D}
Under the assumptions of Theorem~\ref{thm:balanced-norm-1D} there is $\lambda _0 > 0$ such that 
for every $\lambda \in (0,\lambda _0)$ there are constants $C$, $\sigma > 0$ such that 
for all $\varepsilon \in (0,1]$, $p \in \mathbb{N}$
$$
\|u - u_{FEM}\|_{\infty,I} \leq C e ^{-\sigma p}. 
$$
\end{cor}
\begin{proof}
We first observe that standard inverse estimates yield the result when 
$\lambda p \varepsilon \ge 1/4$, in which case the mesh consists of a single element. 
Let us therefore consider the 3-element case $\lambda p \varepsilon < 1/4$.
Using the boundary condition at $x = 0$ we can write 
\begin{equation*}
\left\vert u(x)-u_{FEM}(x)\right\vert =\left\vert \int_{0}^{x}\left(
u(t)-u_{FEM}(t)\right) ^{\prime }dt\right\vert .
\end{equation*}
Assume first that $x\in (0,\lambda p\varepsilon ].$ Then by the Cauchy-Schwarz inequality
and (\ref{eq:balanced_estimate-1D}) 
\[
\left\vert u(x)-u_{FEM}(x)\right\vert  \leq 
\sqrt{\lambda p\varepsilon }\left( C \varepsilon^{-1/2}e^{-\sigma p}\right)  
\leq C\sqrt{\lambda p}e^{-\sigma p}.
\]
The same technique works if $x \in [1-\lambda p \varepsilon, 1)$. For 
$x \in [\lambda p \varepsilon, 1-\lambda p \varepsilon]$, we write 
with the approximation $I_p u$ of Lemma~\ref{lemma:1D-approximation}
and the triangle inequality 
$|u(x) - u_{FEM}(x)| \leq |u(x) - I_p u(x)| + |I_p u(x) - u_{FEM}(x)|$. 
Lemma~\ref{lemma:1D-approximation} takes  care of $|u(x) - I_p u(x)|$ while
$|I_p u(x) - u_{FEM}(x)|$ is treated with 
the 
standard polynomial inverse estimate $\|I_p u - u_{FEM} \|_{\infty,[\lambda p \varepsilon,1-\lambda p \varepsilon]} 
\leq C p^2 \|I_p u - u_{FEM}\|_{0,I}$ and the energy estimate of 
Proposition~\ref{prop:melenk97}. 
\end{proof}

The proof of Theorem~\ref{thm:balanced-norm-1D} is done in two steps: 
First, in Section~\ref{sec:reduction-to-H1-stability-1D} we reduce the analysis to
an $H^1$-stability analysis of a projection operator ${\mathcal P}_0$ that is closely connected with
the reduced/limit problem. Next, we recognize that polynomial inverse estimates will be needed for 
the $H^1$-stability analysis. In order to minimize the adverse impact of small elements of 
size $O(\varepsilon p)$ on inverse estimates, we work with a decomposition of the space $S(\lambda,p)$ 
into global polynomials and polynomials supported by the small elements near the boundary. 
Section~\ref{sec:stable-decompositions-1D} provides the necessary
strengthened Cauchy-Schwarz inequality, and Lemma~\ref{lemma:almost-orthogonal} formulates the 
$H^1$-stability results for ${\mathcal P}_0$. 
Finally, in Section~\ref{sec:conclusion-of-proof-1D} we conclude the 
proof of Theorem~\ref{thm:balanced-norm-1D}. 
\subsubsection{Reduction to an $H^1$-stability analysis for a reduced problem}
\label{sec:reduction-to-H1-stability-1D}
%
Since the desired estimate in the ``asymptotic'' case $\lambda p \varepsilon \ge 1/4$ is easily shown 
(see the formal proof of Theorem~\ref{thm:balanced-norm-1D} at the end of the section) we will 
focus in the following analysis on the 3-element case, i.e., $\lambda p \varepsilon < 1/4$. 

We begin by defining the bilinear form
\begin{equation}
\label{B0}
{\mathcal{B}}_{0}\left( u,v\right) =\left\langle bu,v\right\rangle _{I},
\end{equation}
corresponding to the reduced/limit problem. We also introduce the operator 
${\mathcal{P}}_{0}:L^{2}(I)\rightarrow S_0(\lambda ,p)$ by the orthogonality 
condition\footnote{Note the subtle point that $S_0(\lambda ,p)\subset H_{0}^{1}(I)$; in contrast,
the reduced problem doesn't involve boundary conditions.} 
\begin{equation}
{\mathcal{B}}_{0}\left( u-{\mathcal{P}}_{0}u,v\right) =0 \quad  \forall \; v\in S_0(\lambda ,p).  
\label{Pi0}
\end{equation}
Then, by Galerkin orthogonality satisfied by $u-u_{FEM}$ (with respect to the bilinear form 
${\mathcal B}_\varepsilon$) and by $u - {\mathcal P}_0 u$ (with respect to the bilinear  form 
${\mathcal B}_0$)
we have
\begin{eqnarray}
\label{eq:galerkin-orthogonality-1D}
\left\Vert u_{FEM}-{\mathcal{P}}_{0}u\right\Vert _{E,I}^{2} &=&
{\mathcal{B}}_{\varepsilon }\left( u_{FEM}-{\mathcal{P}}_{0}u,u_{FEM}-{\mathcal{P}}_{0}u\right)  \\
\nonumber 
&=&{\mathcal{B}}_{\varepsilon }\left( u-{\mathcal{P}} _{0}u,u_{FEM}-{\mathcal{P}} _{0}u\right) \\
\nonumber 
&=&\varepsilon ^{2}\left\langle \left( u-{\mathcal{P}}_{0}u\right)^{\prime},\left( u_{FEM}-
{\mathcal{P}} _{0}u\right) ^{\prime }\right\rangle _{I} \\
&\leq& \varepsilon^2 \Vert \left( u-{\mathcal{P}}_{0}u\right)^{\prime} \Vert_{0,I} 
\Vert \left( u_{FEM}-{\mathcal{P}}_{0}u\right)^{\prime} \Vert_{0,I}. \notag
\end{eqnarray}
Hence
\begin{equation*}
\varepsilon \left\Vert \left( u_{FEM}-{\mathcal{P}} _{0}u\right) ^{\prime }\right\Vert
_{0,I}\leq \left\Vert u_{FEM}-{\mathcal{P}} _{0}u\right\Vert _{E,I}\leq \varepsilon
\left\Vert \left( u-{\mathcal{P}} _{0}u\right) ^{\prime }\right\Vert _{0,I}.
\end{equation*}
The triangle inequality will then allow us to infer from this the exponential convergence
result (\ref{eq:balanced_estimate-1D}) provided we can show that 
\begin{equation*}
\left\Vert \left( u-{\mathcal{P}} _{0}u\right)^{\prime }\right\Vert _{0,I} \leq 
C\varepsilon ^{-1/2}e^{-\sigma p},
\end{equation*}
for some $C$ and $\sigma >0$ independent of $\varepsilon$
and $p$. This calculation shows that we have to study the $H^{1}$-stability of the 
operator ${\mathcal{P}}_{0}$ on {\emph{Spectral Boundary Layer meshes}}. 
This is achieved in Lemma~\ref{lemma:almost-orthogonal}. Subsequently in 
Lemma~\ref{lemma:estimate-Pi0u-scalar}, we control $\|(u - {\mathcal P}_0 u)^\prime\|_{0,I}$. 

\subsubsection{Stable decompositions of the spaces $S(\lambda,p)$}
\label{sec:stable-decompositions-1D}
Asymptotic expansions are a tool to decompose the solution $u$ into
components on the different length scales. We need to mimick this on the
discrete level for ${\mathcal{P}} _{0}u$. We define (implicitly assuming $\lambda p \varepsilon <1/4$)
the layer region
$$
I_\varepsilon:= [0,\lambda p \varepsilon] \cup [1-\lambda p \varepsilon,1]
$$
and the following two subspaces of $S(\lambda ,p)$: 
\begin{eqnarray}
\label{S1}
S_{1}& = & {\mathcal S}^{p}(\Delta ), \qquad \Delta =\{0,1\},  \\
\label{Sepsilon}
S_{\varepsilon }&=&\{u\in S(\lambda ,p)\,:\,\operatorname*{supp} u\subset I_\varepsilon\}. 
\end{eqnarray}
Note that the spaces $S_{1}$ and $S_{\varepsilon }$ do not carry any
boundary conditions at the endpoints of $I$ -- this is a reflection of the fact
that the reduced problem does not satisfy the homogeneous Dirichlet boundary
conditions. It is important for the further developments to observe that for 
the three-element mesh of sufficiently small $\lambda p\varepsilon$, there holds 
${S}(\lambda ,p)=S_{1}\oplus S_{\varepsilon }$. 
In other words, each $z\in {S} (\lambda ,p)$ has a unique decomposition 
$z=z_{1}+z_{\varepsilon }$ with $z_{1}\in S_{1}$ and $z_{\varepsilon }\in S_{\varepsilon }$, 
when $\lambda p\varepsilon <1/4.$ We also have the inverse estimates 
\begin{eqnarray}
\Vert z^{\prime }\Vert _{0,I} &\leq& 
Cp^{2}\Vert z\Vert _{0,I}\quad \forall  z\in S_{1},  \label{inv_est1} \\
\Vert z^{\prime }\Vert _{0,I} &\leq& C\frac{p^{2}}{\lambda p\varepsilon }
\Vert z\Vert _{0,I} \quad  \forall  z\in S_{\varepsilon },
\label{inv_est2}
\end{eqnarray}
by \cite[Thm.~{3.91}]{schwab98}. Furthermore, we have the following 
strengthened Cauchy-Schwarz inequality:

\begin{lem} [Strengthened Cauchy-Schwarz inequality]
\label{SCS} Let $\mathcal{B}_{0}$ be given by (\ref{B0}). Then, there is a 
constant $C > 0$ depending solely on $\|b\|_{\infty,I}$ and $\inf_{x \in I} b(x)$ such that 
\begin{equation*}
\left\vert {\mathcal{B}}_{0}\left( u,v\right) \right\vert \leq C \min\{1,\sqrt{\lambda 
p\varepsilon }p\} \left\Vert u\right\Vert _{0,I}\left\Vert v\right\Vert
_{0,I_{\varepsilon }}\quad \forall  u\in S_{1},v\in S_{\varepsilon }.
\end{equation*}
\end{lem}
\begin{proof} 
The standard Cauchy-Schwarz inequality yields $|{\mathcal B}_0(u,v)| \leq \|b\|_{\infty,I} \|u\|_{0,I} \|v\|_{0,I}$, 
which accounts for the ``$1$'' in the minimium. 

Let $I_{1}=(0,\delta _{1})$ and $I_{2}=(0,\delta _{2})$ be two
intervals with $\delta _{1}<\delta _{2}$. Consider polynomials $\pi _{1}$
and $\pi _{2}$ of degree $p$. Then, using an inverse inequality \cite[{eq. (3.6.4)}]{schwab98}, 
\begin{equation*}
\left\vert \int_{I_{1}}\pi _{1}(x)\pi _{2}(x)\,dx\right\vert \leq 
\int_{I_{1}} \left\vert \pi _{1}(x) \right\vert \left\vert \pi _{2}(x) \right\vert\,dx \leq 
C\sqrt{\frac{\delta_{1}}{\delta_{2}}}p\Vert \pi_{1}\Vert _{0,I_{2}}\Vert \pi_{2}\Vert _{0,I_{1}}.
\end{equation*}
The result follows by taking $\delta _{1}=\lambda p\varepsilon$, $\delta _{2}=1$. 
\end{proof}
 
As already mentioned, since ${S}(\lambda ,p)=S_{1}\oplus S_{\varepsilon }$ 
when $\lambda p \varepsilon <1/4,$ we can uniquely decompose ${\mathcal{P}} _{0}u$ into 
components in $S_{1}$ and $S_{\varepsilon }$. The Strengthened Cauchy-Schwarz inquality
of Lemma~\ref{SCS} 
allows us to quantify the size of these contributions: 

\begin{lem}[stability of ${\mathcal P}_0$]
\label{lemma:almost-orthogonal} There exist constants $C$, $c>0$
depending solely on $\inf_{x \in I} b(x) > 0$ and $\|b\|_{\infty,I}$ such that the 
following is true
under the assumption
\begin{equation}
\sqrt{\lambda p\varepsilon }p\leq c:
\label{eq:discrete-scale-resolution-scalar}
\end{equation}
For each $z\in L^{2}(I)$, the (unique) decomposition of
\begin{equation*}
{\mathcal{P}}_{0}z=z_{1}+z_{\varepsilon}
\end{equation*}
into the components $z_1 \in S_{1}$ and $z_{\varepsilon} \in S_{\varepsilon }$ satisfies 
\begin{eqnarray}
\Vert z_{1}\Vert _{0,I} &\leq &C\Vert z\Vert _{0,I},  \label{z1} \\
\Vert z_{\varepsilon }\Vert _{0,I} &\leq & 
C\{\Vert z\Vert _{0,I_{\varepsilon}}+\sqrt{\lambda p\varepsilon }p\Vert z\Vert _{0,I}\}.  
\label{zepsilon}
\end{eqnarray}
Furthermore, 
\begin{eqnarray}
\Vert z_{1}^\prime\Vert _{0,I} &\leq &C p^2\Vert z\Vert _{0,I},  
\label{z1-H1} \\
\Vert z_{\varepsilon }^\prime\Vert _{0,I} &\leq & 
C\left\{ \frac{p^2}{\lambda p \varepsilon} \Vert z\Vert _{0,I_{\varepsilon}}
+(\lambda p\varepsilon )^{-1/2} p^3\Vert z\Vert _{0,I}\right\}.  
\label{zepsilon-H1}
\end{eqnarray}
\end{lem}
\begin{proof} 
Before we start with the proof of (\ref{z1}), (\ref{zepsilon}), we mention 
that (\ref{z1}) follows by fairly standard arguments. Indeed, 
the smallness assumption (\ref{eq:discrete-scale-resolution-scalar}) on $c$ implies the strengthened Cauchy-Schwarz
inequality by Lemma~\ref{SCS}, and for this setting, it is well-known that the contributions $z_1$ and $z_\varepsilon$
can be controlled in terms of the constant of the strengthened Cauchy-Schwarz inequality and 
$\|{\mathcal P}_0 z\|_{0,I}$. 
\iftechreport (See Appendix for details.)
\fi
This result produces (\ref{z1}) but not (\ref{zepsilon}), for which we need
to refine the standard analysis. This is done below. In the interest of completeness, we will nevertheless 
present a proof for both (\ref{z1}), (\ref{zepsilon}). 

Write ${\mathcal P}_0 z = z_1 + z_\varepsilon$ with $z_1 \in S_1$ and $z_\varepsilon \in S_\varepsilon$. 
We define the auxiliary function
\begin{equation*}
\psi _{1,\varepsilon }:=\begin{cases}
\left( 1-\frac{x}{\lambda p\varepsilon }\right) ^{p} & \mbox{ if $x\in \lbrack
0,\lambda p\varepsilon ]$} \\ 
0 & \mbox{otherwise.}
\end{cases}
\end{equation*}
Then 
$\text{supp }\psi_{1,\varepsilon }\subset \lbrack 0,\lambda p\varepsilon ],\psi_{1,\varepsilon }(0)=1$ 
and $\left\Vert \psi_{1,\varepsilon }\right\Vert _{0,I_{\varepsilon }} \sim p^{-1/2}\sqrt{\lambda 
p\varepsilon }$. For the right endpoint we define 
$\psi_{2,\varepsilon }(x):=\psi _{1,\varepsilon }(1-x), x \in [1-\lambda p \varepsilon, 1]$.
We also define
\begin{equation*}
\widetilde{z}_{\varepsilon }:=
z_{\varepsilon }+\psi _{1,\varepsilon}z_{1}(0)+\psi _{2,\varepsilon }z_{1}(1),
\end{equation*}
and note that ${\mathcal{P}}_0 z \in S_0(\lambda ,p)$. Thus, 
$(z_1 + z_\varepsilon)|_{\partial I}= 0$ so that 
$\widetilde{z}_{\varepsilon }\in S_{\varepsilon}\cap H_{0}^{1}(I)\subset S_{\varepsilon }\cap S_0(\lambda ,p)$. 
Utilizing the inverse estimate \cite[Thm.~{3.92}]{schwab98} 
\begin{equation*}
\left\Vert \pi \right\Vert _{\infty, I} \leq C p \left\Vert \pi
\right\Vert _{0,I} \; \forall \; \pi \in S_{1},
\end{equation*}
we arrive at 
\begin{equation*}
\left\Vert \widetilde{z}_{\varepsilon }\right\Vert _{0,I}=\left\Vert 
\widetilde{z}_{\varepsilon }\right\Vert _{0,I_{\varepsilon }}\leq C\left\{
\left\Vert z_{\varepsilon }\right\Vert _{0,I_{\varepsilon }}+p^{1/2}\sqrt{
\lambda p\varepsilon }\left\Vert z_{1}\right\Vert _{0,I}\right\} .
\end{equation*}

The representation ${\mathcal{P}}_{0}z=z_{1}+z_{\varepsilon }\in
S_0(\lambda ,p)$ also implies 
\begin{eqnarray}
{\mathcal{B}}_{0}(z_{1},v_{1})+{\mathcal{B}}_{0}(z_{\varepsilon },v_{1}) &=&
{\mathcal{B}}_{0}({\mathcal{P}} _{0}z,v_{1}) \quad \forall \text{ }v_{1}\in S_{1}, 
\label{1} \\
{\mathcal{B}}_{0}(z_{1},v_{\varepsilon })+{\mathcal{B}}_{0}(z_{\varepsilon
},v_{\varepsilon }) &=&{\mathcal{B}}_{0}({\mathcal{P}} _{0}z,v_{\varepsilon })=
{\mathcal{B}}_{0}(z,v_{\varepsilon })\quad \forall \text{ }v_{\varepsilon
}\in S_{\varepsilon }\cap S_0(\lambda ,p),  \label{2}
\end{eqnarray}
where in (\ref{2}) we used the fact that ${\mathcal{P}} _{0}$ is the ${\mathcal B}_0$--projection 
onto $S_0(\lambda ,p)$. Taking $v_{1}=z_{1}$ in (\ref{1}) and 
$v_{\varepsilon }=\widetilde{z}_{\varepsilon }\in S_{\varepsilon }\cap S_0(\lambda ,p)$ 
in (\ref{2}) yields, together with the Strengthened Cauchy Schwarz inequality of Lemma~\ref{SCS},
\begin{subequations}
\begin{eqnarray}
\Vert z_{1}\Vert _{0,I}^{2} &\leq& C \{ \Vert {\mathcal{P}}_{0}z\Vert _{0,I}\Vert
z_{1}\Vert _{0,I}+p\sqrt{\lambda p\varepsilon }\Vert z_{\varepsilon }\Vert
_{0,I}\Vert z_{1}\Vert _{0,I} \}, 
\label{lemma:stable-discrete-decomposition-scalar-500} \\
\Vert z_{\varepsilon }\Vert _{0,I}^{2} &\leq& C \{ \Vert z\Vert
_{0,I_{\varepsilon }}\Vert \widetilde{z}_{\varepsilon }\Vert
_{0,I_{\varepsilon }}+p\sqrt{\lambda p\varepsilon }\Vert \widetilde{z}
_{\varepsilon }\Vert _{0,I}\Vert z_{1}\Vert _{0,I}+\Vert z_{\varepsilon
}\Vert _{0,I}\Vert z_{1}\Vert _{0,I}\sqrt{\lambda p\varepsilon }p^{1/2} \}
\notag \\
&\leq& C \{ \Vert z_{\varepsilon }\Vert _{0,I}\left[ \Vert z\Vert
_{0,I_{\varepsilon }}+p\sqrt{\lambda p\varepsilon }\Vert z_{1}\Vert _{0,I}+
\sqrt{\lambda p \varepsilon}p^{1/2}\Vert z_{1}\Vert _{0,I}\right]  \notag \\
&&\qquad \mbox{}+\left[ \Vert z\Vert _{0,I_{\varepsilon }}+p\sqrt{\lambda 
p\varepsilon }\Vert z_{1}\Vert _{0,I}\right] \sqrt{\lambda p\varepsilon }
p^{1/2}\Vert z_{1}\Vert _{0,I} \}.
\label{lemma:stable-discrete-decomposition-scalar-500-b}
\end{eqnarray}
\end{subequations}
Estimating generously $\sqrt{\lambda p\varepsilon }p^{1/2}\leq \sqrt{\lambda p\varepsilon }p$ 
and using an appropriate Young inequality in 
(\ref{lemma:stable-discrete-decomposition-scalar-500-b}) we get 
\begin{subequations}
\begin{eqnarray}
\Vert z_{1}\Vert _{0,I} &\leq& C \{ \Vert {\mathcal{P}}_{0}z\Vert _{0,I}+p\sqrt{\lambda 
p\varepsilon }\Vert z_{\varepsilon }\Vert _{0,I}\},
\label{lemma:stable-discrete-decomposition-scalar-600} \\
\Vert z_{\varepsilon }\Vert _{0,I} &\leq& C \{ \Vert z\Vert
_{0,I_{\varepsilon }}+p\sqrt{\lambda p\varepsilon }\Vert z_{1}\Vert _{0,I}
\} .  \label{lemma:stable-discrete-decomposition-scalar-600-b}
\end{eqnarray}
\end{subequations}
Inserting (\ref{lemma:stable-discrete-decomposition-scalar-600-b}) in 
(\ref{lemma:stable-discrete-decomposition-scalar-600}), assuming that 
$\sqrt{\lambda p\varepsilon }p$ is sufficiently small and using the stability 
$\Vert {\mathcal{P}}_{0}z\Vert _{0,I}\leq C \Vert z\Vert _{0,I}$ gives 
$\Vert z_{1}\Vert_{0,I}\leq C\Vert z\Vert _{0,I}$. Inserting this bound in 
(\ref{lemma:stable-discrete-decomposition-scalar-600-b}) concludes the proof
of (\ref{z1}) and (\ref{zepsilon}). Finally, the proof (\ref{z1-H1}), (\ref{zepsilon-H1})
follows from a further application of the standard polynomial inverse estimates 
(\ref{inv_est1}), (\ref{inv_est2}). 
\end{proof}
\subsubsection{Conclusion of the proof of Theorem~\ref{thm:balanced-norm-1D}}
\label{sec:conclusion-of-proof-1D}
We are now in the position to prove the following
\begin{lem}
\label{lemma:estimate-Pi0u-scalar}
Assume (\ref{analytic_data}). Let $u$ be the solution of (\ref{BuvFv}) and 
let $\lambda_0$ be given by Lemma~\ref{lemma:1D-approximation}. 
Let $\lambda \in (0,\lambda_0]$ and assume that $\lambda$, $p$, $\varepsilon$ 
satisfy (\ref{eq:discrete-scale-resolution-scalar}).  
Then there exist constants $C$, $\beta > 0$ (independent of $\varepsilon\in(0,1]$ 
and $p\in\N$ but dependent on $\lambda$) such that  
\begin{equation}
\label{eq:key-estimate-scalar}
\|(u - {\mathcal{P}}_0 u)^\prime \|_{0,I} \leq C \varepsilon^{-1/2} e^{-\beta p}.
\end{equation}
\end{lem}
\begin{proof} 
Recall that only the case $\lambda p \varepsilon < 1/4$ is of interest. 
By Lemma~\ref{lemma:1D-approximation} we can find 
an approximation $I_{p}u\in S_0(\lambda ,p)$ with 
\begin{equation}
\Vert u-I_{p}u\Vert _{0,I}+\sqrt{\varepsilon }\Vert (u-I_{p}u)^{\prime
}\Vert _{0,I}\leq Ce^{-\beta p}.  \label{eq:lemma:estimate-Pi0u-scalar-10}
\end{equation}
We stress that, while the estimate (\ref{eq:lemma:1D-approximation-10}) is explicit 
in the parameter $\lambda$, we have absorbed this dependence here in the constants $C$ and $\beta$  
for simplicity of exposition. 

Since ${\mathcal{P}}_0$ is a projection on $S_0(\lambda ,p)$ and $I_p u \in S_0(\lambda,p)$, 
we can write $u- {\mathcal{P}}_0 u=u-I_{p}u-{\mathcal{P}}_0 (u-I_{p}u)$. 
The first term, $u - I_p u$, is already treated in (\ref{eq:lemma:estimate-Pi0u-scalar-10}). 
For the second term, ${\mathcal{P}}_0 (u-I_{p}u)\in S_0(\lambda ,p)$, we decompose 
${\mathcal{P}}_0 (u-I_{p}u)=z_{1}+z_{\varepsilon }$ and use the estimates 
(\ref{z1-H1}), (\ref{zepsilon-H1}) of Lemma~\ref{lemma:almost-orthogonal} to get 
\begin{eqnarray*}
\Vert z_{1}^{\prime }\Vert _{0,I} &\lesssim &
p^{2}\Vert u-I_{p}u\Vert _{0,I} \leq Ce^{-\beta p}, \\
\Vert z_{\varepsilon }^{\prime }\Vert _{0,I} &\lesssim &
\frac{p^{2}}{\lambda p\varepsilon }\left[ \Vert u-I_{p}u\Vert
_{0,I_{\varepsilon }}+\sqrt{\lambda p\varepsilon }p\Vert u-I_{p}u \Vert
_{0,I}\right] .
\end{eqnarray*}
There are several possible ways to treat the term 
$\Vert (u-I_{p}u)\Vert_{0,I_{\varepsilon }}$. A rather generous approach exploits the fact that 
$(u-I_{p}u)(0)=(u-I_{p}u)(1)=0$ so that we use $z(x)=\int_{0}^{x}z^{\prime}(t)\,dt$ 
and obtain 
\begin{equation*}
\Vert u-I_{p}u\Vert _{0,I_{\varepsilon }}\leq C\lambda p\varepsilon \Vert
(u-I_{p})^{\prime }\Vert _{0,I_{\varepsilon }}.
\end{equation*}
Hence, 
\begin{equation*}
\Vert z_{\varepsilon }^{\prime }\Vert _{0,I}\lesssim 
\frac{p^{2}}{\lambda p \varepsilon}\left[\lambda p \varepsilon \Vert
(u-I_{p}u)^{\prime }\Vert _{0,I_{\varepsilon }}+\sqrt{\lambda p\varepsilon }
p\Vert u-I_{p}u\Vert _{0,I}\right] \lesssim \varepsilon ^{-1/2}e^{-\beta p}.
\end{equation*}
\end{proof}

\begin{numberedproof}{Theorem~\ref{thm:balanced-norm-1D}} 
In view of $\|u - u_{FEM}\|_{0,I} \leq C \|u - u_{FEM}\|_{E,I} \leq C e^{-\sigma p}$
by Proposition~\ref{prop:melenk97}, 
we focus on the control
of $\sqrt{\varepsilon} \|(u - u_{FEM})^\prime\|_{0,I}$. 
We distinguish two cases: 

{\em Case 1:} Assume that (\ref{eq:discrete-scale-resolution-scalar}) is satisfied. Then 
(\ref{eq:key-estimate-scalar}) yields the result. 

{\em Case 2:} Assume that $\sqrt{\lambda p \varepsilon} p \ge c$ for the 
constant $c$ appearing in (\ref{eq:discrete-scale-resolution-scalar}). 
Then $\varepsilon^{-1/2} \leq c^{-1} p^{3/2} \lambda ^{1/2}$ so that 
\begin{eqnarray*}
\sqrt{\varepsilon} \|(u - u_{FEM})^\prime\|_{0,I} &\leq& 
\varepsilon^{-1/2} \|u - u_{FEM}\|_{E,I} 
\leq c^{-1} \lambda^{1/2} p^{3/2} \|u - u_{FEM}\|_{E,I} 
\lesssim e^{-\sigma p},
\end{eqnarray*}
which concludes the proof.
\end{numberedproof}
\subsection{Numerical example}
To illustrate the theoretical findings presented above, we show in Figure
\ref{fig:1D-example}
the results of numerical computations for the following problem: 
\begin{eqnarray*}
-\varepsilon^2  u^{\prime \prime }(x)+u(x) &=&\left( x+\frac{1}{2}\right)
^{-1},\quad x\in (0,1), \\
u(0) &=&u(1)=0.
\end{eqnarray*}
We use the \emph{Spectral Boundary Layer mesh} $\Delta_{BL}(\lambda,p)$  with $\lambda = 1$ 
and polynomials of 
degree $p$ which we increase from $1$ to $5$ to improve accuracy.  
We select $\varepsilon = 10^{-j}$, $j=4,...,8$. We 
note $\operatorname*{dim} S_0(\lambda ,p) = 2+3(p-1)$.  
Since no exact solution is available, we use a reference solution to estimate the error. In 
Fig.~\ref{fig:1D-example}, we present the error in the 
balanced norm (\ref{eq:balanced-norm-1D}) versus the polynomial degree $p$ as well as 
the error $\varepsilon^{1/2} \|(u - u_{FEM})^\prime\|_{0,I}$ and the $L^2$-error. The error
curves are on top of each other, which supports the robust exponential convergence 
in the balanced norm. 

\begin{figure}[ht]
\par
\begin{center}
\includegraphics[width=0.5\textwidth]{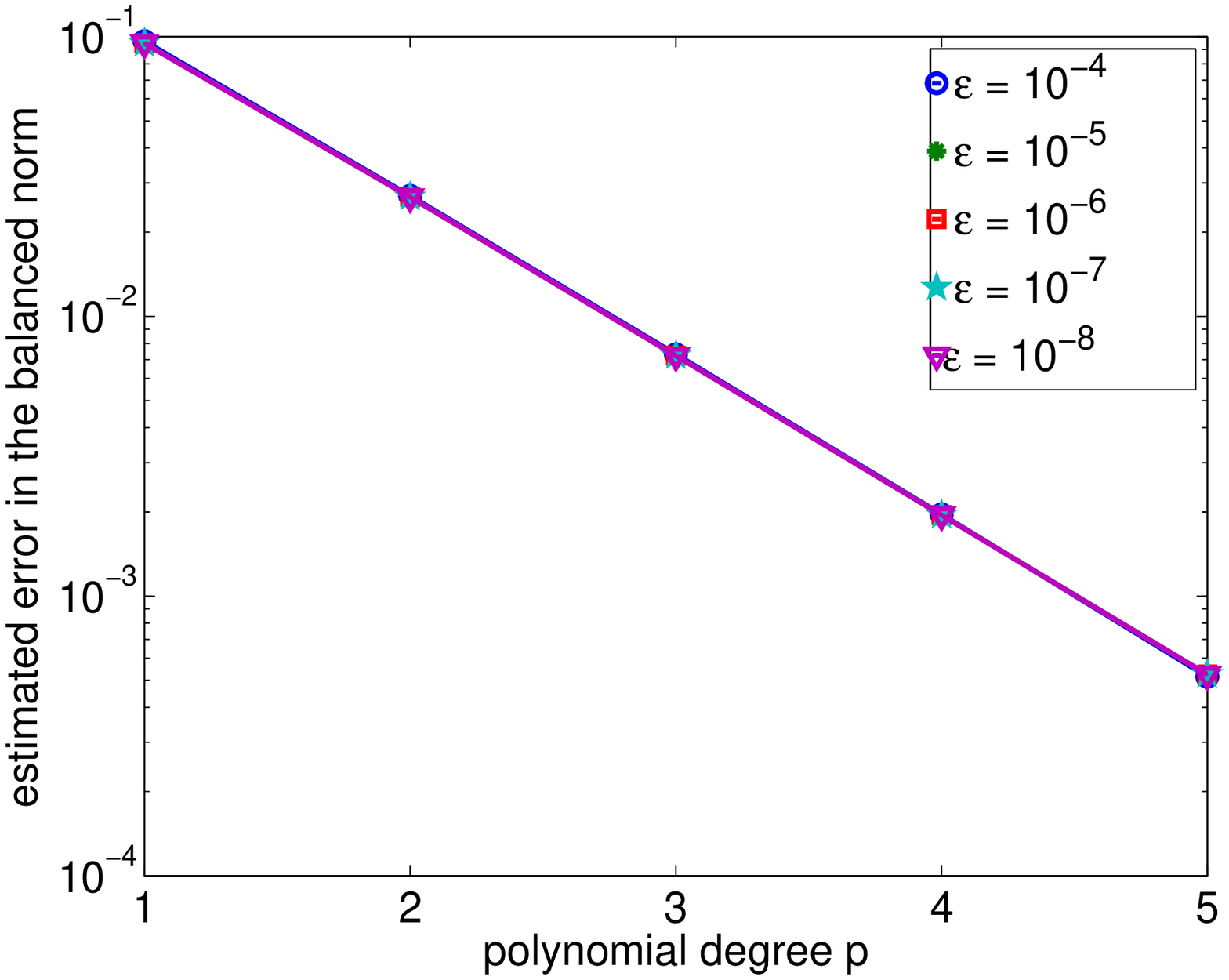}
\par
\includegraphics[width=0.45\textwidth]{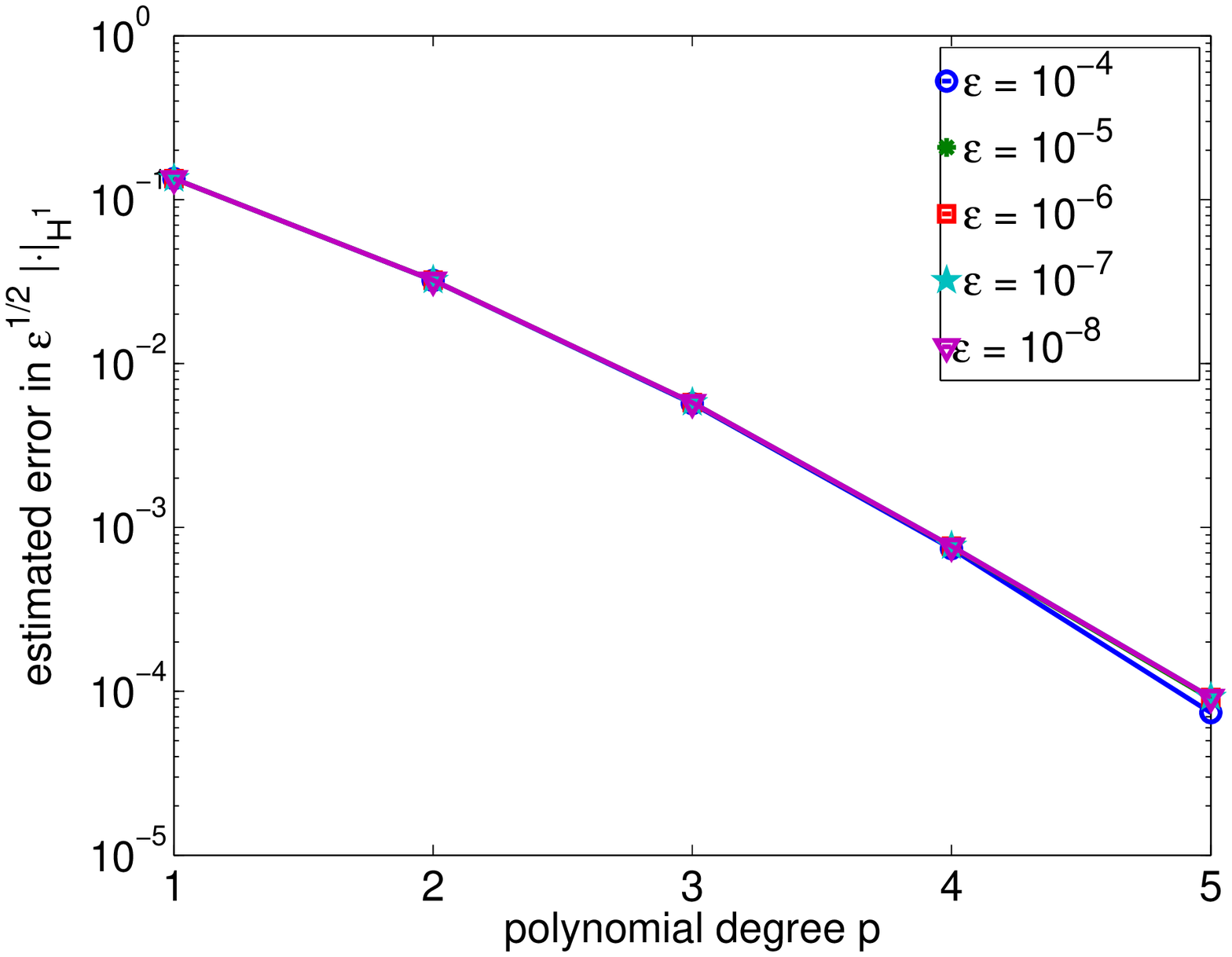}
\includegraphics[width=0.45\textwidth]{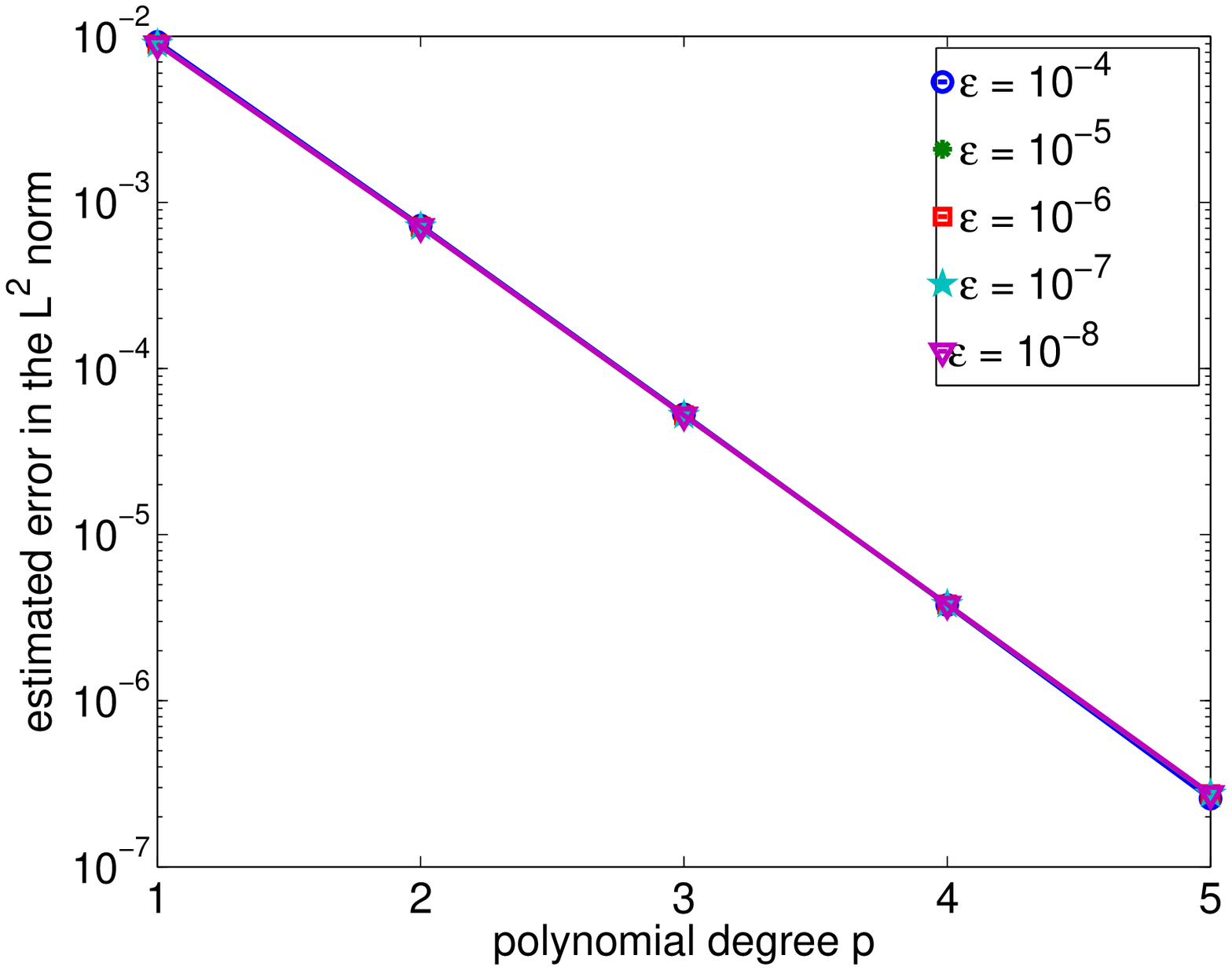}
\end{center}
\caption{\label{fig:1D-example} Convergence on \emph{Spectral Boundary Layer meshes.} 
Top: convergence in the balanced norm. Bottom left: error $\varepsilon^{1/2} \|(u - u_{FEM})^\prime\|_{0,I}$ versus $p$. Bottom right: convergence in $L^2$.}
\end{figure}

\section{The two-dimensional case}
\label{2D} 
The ideas of the previous section carry over to the two-dimensional
case. We consider the following boundary value problem: Find $u$ such that 
\begin{subequations}
\label{eq:2D-problem}
\begin{eqnarray}
-\varepsilon ^{2}\Delta u+b u &=&f \quad \text{ in }\Omega \subset 
\mathbb{R}^{2},  \label{bvp2D} \\
u &=&0\quad \text{ on }\partial \Omega ,  \label{BC2D}
\end{eqnarray}
\end{subequations}
where $\varepsilon \in (0,1]$, and the functions $b$, $f$ are given with $b>0$ on $\overline{\Omega}$. We assume that
the data of the problem is analytic, i.e., $\partial \Omega $ is an analytic
curve and that there exist constants $C_{f}$, $\gamma _{f}$, $C_{b}$, $\gamma _{b}$, $c_b >0$ such that 
\begin{equation}
\left\{ 
\begin{array}{c}
\left\Vert \nabla ^{n}f\right\Vert _{\infty ,\Omega }\leq C_{f}\gamma
_{f}^{n}n!\quad \forall \;n\in \mathbb{N}_{0}, \\ 
\left\Vert \nabla ^{n}b\right\Vert _{\infty ,\Omega }\leq C_{b}\gamma
_{b}^{n}n!\quad \forall \;n\in \mathbb{N}_{0}, \\
\inf_{x \in \Omega} b(x) \ge c_b > 0 .
\end{array}
\right.   \label{analytic_data_2D}
\end{equation}
The variational formulation of (\ref{bvp2D}), (\ref{BC2D}) reads: Find 
$u\in H_{0}^{1}\left( \Omega \right)$ such that 
\begin{equation}
{\mathcal{B}}_{\varepsilon }(u,v):=\varepsilon^2 \left \langle \nabla u,\nabla v\right\rangle_\Omega + 
\left\langle b u,v\right\rangle_{\Omega} =F(v):= 
\left\langle f,v\right\rangle_{\Omega}  \quad \forall  v\in H_{0}^{1}\left(\Omega \right),  
\label{BuvFv2D}
\end{equation}
where $\left\langle {\cdot ,\cdot }\right\rangle_{\Omega} $ denotes the usual $L^{2}(\Omega )$ 
inner product. As in 1D, the energy norm $\|\cdot\|_{E,\Omega}$ and the balanced norm 
$\|\cdot\|_{balanced,\Omega}$ are defined by 
\begin{eqnarray}
\label{eq:energy-norm-2D}
\|v\|^2_{E,\Omega} &:=& {\mathcal B}_\varepsilon(v,v), \\  
\label{eq:balanced-norm-2D}
\|v\|^2_{balanced,\Omega} &:=& \|v\|^2_{0,\Omega} + \varepsilon \|\nabla v\|^2_{0,\Omega}
\end{eqnarray}
The discrete version of (\ref{BuvFv2D}) reads: find 
$u_{FEM}\in V_{N}\subset H_{0}^{1}\left( \Omega \right)$ such that (\ref{BuvFv2D}) 
holds for all $v\in V_{N}\subset H_{0}^{1}\left( \Omega \right)$, with $u$ replaced by 
$u_{FEM}$, where the subspace $V_{N}$ will be defined shortly. 
\subsection{Meshes and spaces}
\label{sec:2D-meshes}
Concerning the meshes and the $hp$-FEM space based on these meshes, we adopt the simplest 
case that generalizes our 1D analysis to 2D: The elements are (curvilinear) quadrilaterals 
and the needle elements required to resolve the boundary layer are obtained as mappings 
of needle elements of a reference configuration. 
This approach is discussed in more detail in \cite[Sec.~{3.1.2}]{MelenkSchwab} and expanded 
as the notion of ``patchwise structured meshes'' in \cite[Sec.~{3.3.2}]{mB}. 

Our $hp$-FEM spaces have the following general structure: 
Let $\Delta =\left\{ \Omega _{i}\right\} _{i=1}^{N}$ be a mesh
consisting of curvilinear quadrilaterals $\Omega _{i}$, $i=1,\ldots,N$, subject to the usual
restrictions (see, e.g., \cite{MelenkSchwab}) and associate with each $\Omega_{i}$ 
a bijective, Lipschitz continuous (further smoothness assumptions are imposed below) 
element mapping $M_{i}:\Sref\rightarrow \overline{\Omega }_{i},$ 
where $\Sref=[0,1]^{2}$ denotes the usual reference square. With $Q_{p}(\Sref)$ the space of 
polynomials of degree $p$ (in each variable) on $\Sref$, we set 
\begin{eqnarray*}
\mathcal{S}^{p}(\Delta ) &=&\left\{ u\in H^{1}\left( \Omega \right)
:\left. u\right\vert _{\Omega _{i}}\circ M_{i} \in Q_p(\Sref),\quad i=1,...,N\right\},\\
\mathcal{S}_{0}^{p}(\Delta ) &=&\mathcal{S}^{p}(\Delta )\cap
H_{0}^{1}(\Omega ).
\end{eqnarray*}

We now describe the mesh $\Delta$ and the element maps that we will use (see Fig.~\ref{fig:2D-meshes}). 
Our starting point is a {\em fixed} mesh $\Delta_A$ (the subscript ``$A$'' stands for ``asymptotic'') 
consisting of 
curvilinear quadrilateral elements $\Omega_i$, $i=1,\ldots,N^\prime$. 
These elements $\Omega_i$ are the images of the reference square $\Sref = [0,1]^2$ under 
the element maps $M_{A,i}$, $i=1,\ldots,N^\prime$ (we added the subscript ``$A$'' to emphasize 
that they correspond to the asymptotic mesh $\Delta_A$).
They are assumed to satisfy the conditions
(M1)--(M3) of \cite{MelenkSchwab} in order to ensure that the space ${\mathcal S}^p(\Delta_A)$ 
has suitable approximation properties. The element 
maps $M_{A,i}$ are assumed to be analytic with analytic inverse; that is, as in \cite{MelenkSchwab} we require 
for some constants $C_1$, $C_2$, $\gamma > 0$
$$
\|(M_{A,i}^\prime)^{-1} \|_{\infty,\Sref} \leq  C_1, 
\qquad \|D^\alpha M_{A,i}\|_{\infty,\Sref} \leq C_2 \alpha! \gamma^{|\alpha|} 
\qquad \forall \alpha \in \N_0^2, \quad i=1,\ldots,N^\prime . 
$$
We furthermore assume that elements do not have a single vertex on the boundary $\partial\Omega$ but 
only complete, single edges, i.e., the following dichotomy holds: 
\begin{equation}
\mbox{ either } \quad 
\overline{\Omega_i} \cap \partial \Omega = \emptyset \quad \mbox{ or } 
\overline{\Omega_i} \cap \partial \Omega \quad \mbox{ is a single edge of $\Omega_i$}. 
\end{equation}
Edges of curvilinear quadrilaterals are, of course, the images of the edges of $\Sref$ under the element maps. 
For notational convenience, we assume that the edges lying on $\partial\Omega$ are the image of the 
edge $\{0\} \times [0,1]$ 
under the element map. It then follows that these elements have one edge on $\partial\Omega$ and the images 
of the edges $\{y = 1\}$ and $\{y = 0\}$ of $\Sref$ are shared with elements that likewise have one edge on 
$\partial\Omega$. 
For notational convenience, we assume that the elements at the boundary 
are numbered first, i.e., they are the elements $\Omega_i$, $i=1,\ldots,n < N^\prime$. 
For a parameter $\lambda > 0$ and a degree $p \in \N$, the 
boundary layer mesh $\Delta_{BL} = \Delta_{BL}(\lambda,p)$ is defined as follows. 
\begin{defn}[Spectral Boundary Layer mesh $\Delta_{BL}(\lambda,p)$]
\label{SBL-2D}
Given parameters $\lambda > 0$, $p \in \N$, $\varepsilon \in (0,1]$ and the asymptotic 
mesh $\Delta_A$, the mesh $\Delta_{BL}(\lambda,p)$ is defined as follows: 
\begin{enumerate}
\item $\lambda p\varepsilon \geq 1/2.$ In this case we are in the asymptotic
regime, and we use the asymptotic mesh $\Delta _{A}$.  

\item $\lambda p\varepsilon <1/2.$ In this regime, we need to define
so-called needle elements. This is done by splitting the elements 
$\Omega_{i},i=1,\ldots,n$ into two elements $\Omega _{i}^{need}$ and 
$\Omega_{i}^{reg}.$ To that end, split the reference square $\Sref$ into two elements
\begin{equation*}
S^{need}=\left[ 0,\lambda p\varepsilon \right] \times \lbrack 0,1],
\qquad S^{reg}=
\left[ \lambda p\varepsilon ,1\right] \times \lbrack 0,1],
\end{equation*}
and define the elements $\Omega_i^{need}$, $\Omega_i^{reg}$ as the images of these 
two elements under the element 
map $M_{A,i}$ and the corresponding element maps as the concatination of the 
affine maps  
\begin{align*}
& A^{need}: \Sref \rightarrow S^{need}, 
\qquad (\xi,\eta) \rightarrow (\lambda p \varepsilon \xi, \eta), \\
& A^{reg}: \Sref \rightarrow S^{reg}, 
\qquad (\xi,\eta) \rightarrow (\lambda p \varepsilon + (1-\lambda p \varepsilon) \xi, \eta) 
\end{align*}
with the element map $M_{A,i}$, i.e., $M_i^{need} = M_{A,i} \circ A^{need}$ and $M_i^{reg} = M_{A,i} \circ A^{reg}$. 
Explicitly: 
\begin{align*}
\Omega _{i}^{need} &=M_{A,i}\left( S^{need}\right), &\Omega
_{i}^{reg}&=M_{A,i}\left( S^{reg}\right) , \\
M_{i}^{need}(\xi ,\eta ) &=M_{A,i}\left( \lambda p\varepsilon \xi ,\eta
\right), & M_{i}^{reg}(\xi ,\eta )&=M_{A,i}\left( \lambda p\varepsilon
+(1-\lambda p\varepsilon )\xi ,\eta \right) .
\end{align*}
\end{enumerate}
\end{defn}
\begin{figure}[ht]
\label{fig2Dmesh}
\par
\begin{center}
\includegraphics[width=0.4\textwidth]{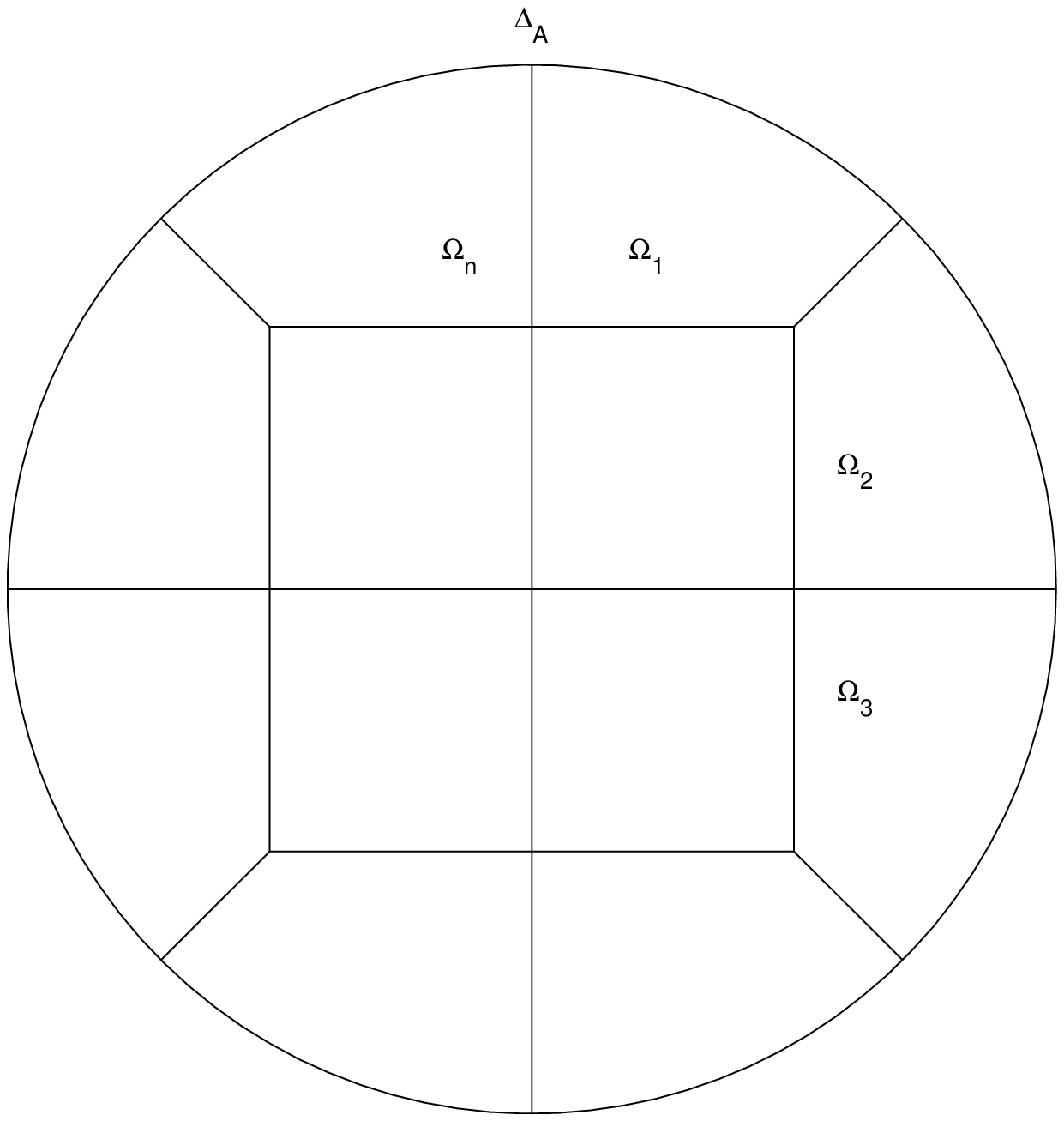} \mbox{ } 
\includegraphics[width=0.4\textwidth]{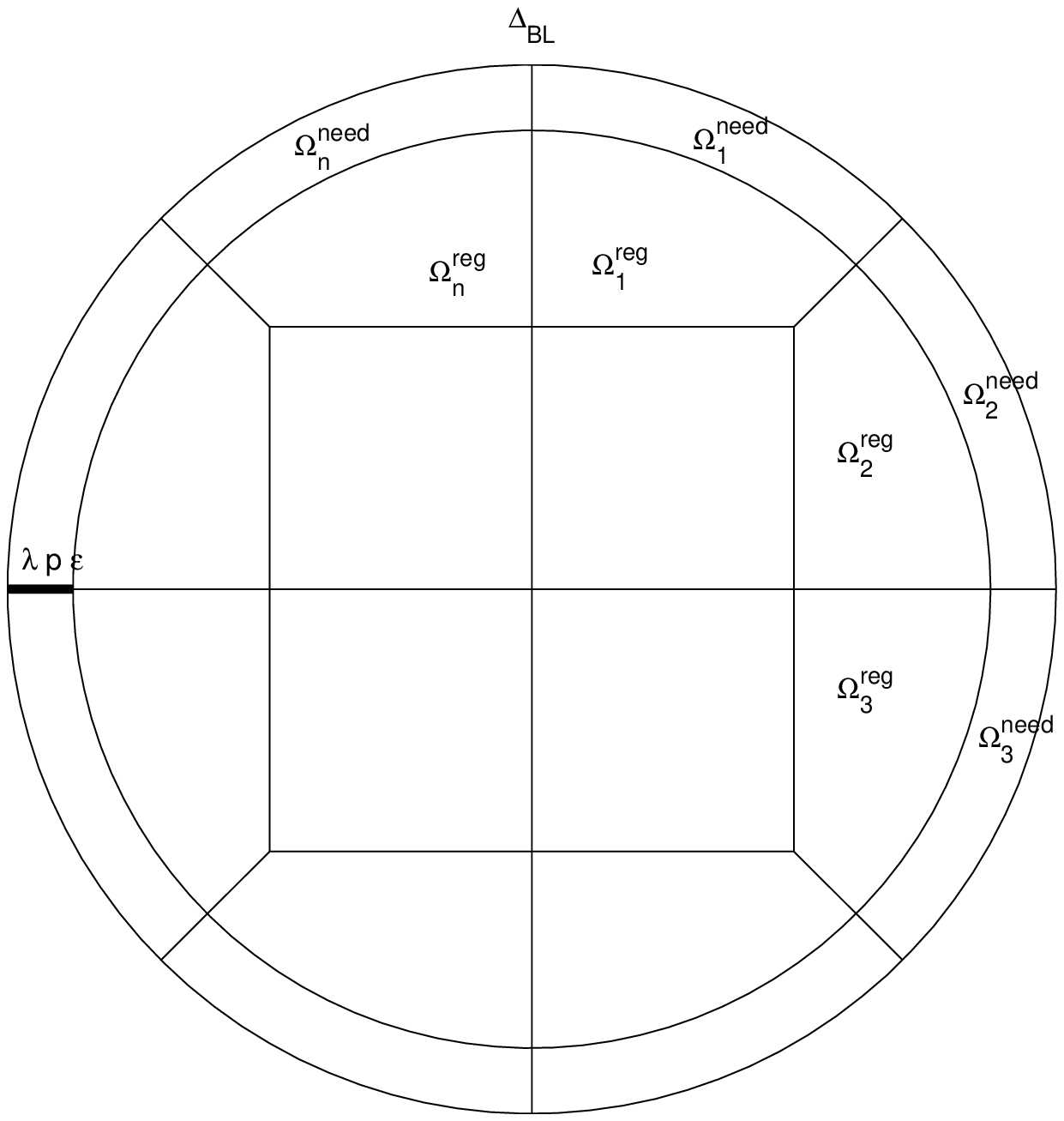}
\end{center}
\caption{\label{fig:2D-meshes} Example of an admissible mesh. Left: asymptotic mesh $\Delta_A$. 
Right: boundary layer  mesh $\Delta_{BL}$.}
\end{figure}
In Figure~\ref{fig:2D-meshes} we show an example of such a mesh construction on the unit circle. 
In total, the mesh $\Delta_{BL}(\lambda,p)$ consists of $N = N^\prime +n$ elements if $\lambda p \varepsilon <1/2$. 

Anticipating that we will need, for the case $\lambda p \varepsilon <1/2$, a decomposition of 
$$
S(\lambda,p):= {\mathcal S}^{p} (\Delta_{BL}(\lambda,p))
$$
into two spaces reflecting the two scales present, we proceed as follows: 
With\ $\Delta _{A}$ the asymptotic (coarse) mesh that resolves the geometry 
we set 
\begin{eqnarray}
\label{S12D}
S_{1}& := &\mathcal{S}^{p}(\Delta _{A}),  \\
\label{Sepsilon2D}
S_{\varepsilon} &:=& \{v \in {\mathcal S}^p(\Delta_{BL}(\lambda,p)) \,|\, 
\operatorname*{supp} v \subset \overline{\Omega}_{\lambda p \varepsilon}\}, 
\end{eqnarray}
where the boundary layer region $\Omega_{\lambda p \varepsilon}$ is defined as 
\begin{equation}
\Omega _{\lambda p\varepsilon }=\overset{n}{\underset{i=1}{\cup }}\Omega_{i}^{need}  . 
\label{Omega_lpe}
\end{equation}
As in the 1D situation, our approximation space 
${\mathcal S}^p(\Delta_{BL}(\lambda,p))$ can be written as a direct sum 
of $S_1$ and $S_\varepsilon$ if $\lambda p \varepsilon <1/2$: 
\begin{lem}
\label{lemma:2D-direct-sum}
Let $\lambda p \varepsilon < 1/2$. Then ${\mathcal S}^p(\Delta_{BL}(\lambda,p))$ is the 
direct sum $S_1 \oplus S_\varepsilon$. Furthermore, we have the inverse estimates 
\begin{eqnarray}
\|u\|_{0,\partial \Omega _{i}} &\leq &Cp\Vert u\Vert _{0,\Omega _{i}}\quad 
\forall  u\in S_{1},\quad i=1,...,N^\prime,  \label{inv2Daa} \\
|u|_{1,\Omega _{i}} &\leq &Cp^{2}\Vert u\Vert _{0,\Omega _{i}}\quad 
\forall  u\in S_{1},\quad i=1,...,N^\prime,  \label{inv2Da} \\
|u|_{1,\Omega _{i}} &\leq &C\frac{p^{2}}{\lambda p\varepsilon }\Vert u\Vert_{0,\Omega _{i}}\quad 
\forall  u\in S_{\varepsilon },i=1,...,n,
\label{inv2Db}
\end{eqnarray}
\end{lem}
\begin{proof} 
The claim that ${\mathcal S}^p(\Delta_{BL}(\lambda,p)) = S_1 \oplus S_\varepsilon$ 
follows from the way $\Delta_{BL}(\lambda,p)$ is constructed. 
Let $z \in {\mathcal S}^p(\Delta_{BL}(\lambda,p))$. Define $z_1 \in S_1$ as follows: 
For the internal elements $\Omega_i$ with $i=n+1,\ldots,N^\prime$ 
take $z_1|_{\Omega_i}:= z|_{\Omega_i}$. For $\Omega_i$, $i \in \{1,\ldots,n\}$, which is 
further decomposed into $\Omega_i^{need}$ and $\Omega_i^{reg}$, we consider the 
pull-back $\widetilde z_i:= z|_{\Omega_i} \circ M_{A,i}$. This pull-back $\widetilde z_i$ 
is a piecewise polynomial on $\Sref = S^{need} \cup S^{reg}$. Define the \emph{polynomial} 
$\widehat z_i \in Q(\Sref)$ on the full reference element $\Sref$ by the condition 
$$
\widehat z_i|_{S^{reg}} = \widetilde z_i|_{S^{reg}}
$$
and then set $z_1|_{\Omega_i}:= \widehat z_i \circ M_{A,i}^{-1}$; that is, the restriction 
$\widetilde z_i|_{S^{reg}}$ is extended polynomially to $\Sref$. 
In this way, the function $z_1$ is 
defined elementwise, and the assumptions on the element maps $M_{A,i}$ of the asymptotic mesh $\Delta_A$
ensure that $z_1 \in H^1(\Omega)$, i.e., $z_1 \in S_1$. Since by construction 
$z|_{\Omega_i^{reg}}  = z_1|_{\Omega_i^{reg}}$ for $i=1,\ldots,n$, we conclude that 
$\operatorname*{supp} (z - z_1) \subset \overline{\Omega}_{\lambda p \varepsilon}$ and therefore 
$z_\varepsilon:= z - z_1 \in S_\varepsilon$. The construction also shows the uniqueness of the decomposition. 

The inverse estimates (\ref{inv2Daa}), (\ref{inv2Da}), (\ref{inv2Db}) can be seen as follows. The estimate
(\ref{inv2Da}) 
 is an easy consequence of the assumptions on the element maps $M_{A,i}$ of the asymptotic 
mesh $\Delta_A$ and the polynomial inverse estimates
\cite[Thm.~{4.76}]{schwab98}. In a similar manner, the inverse estimate (\ref{inv2Daa}), which estimates
the $L^2$-norm on the boundary $\partial\Omega_i$ of $\Omega_i$ by the $L^2$-norm on $\Omega_i$ 
follows from a suitable application of 1D inverse estimates 
(cf. \cite[eqn. (3.6.4)]{schwab98}).

For the estimate (\ref{inv2Db}), we note that for an element $\Omega_i^{need}$, we can 
estimate for any $v \in S_\varepsilon$ again with assumptions on the element maps $M_{A,i}$  
$$
\|\nabla v\|_{0,\Omega_i^{need}} \sim \|\nabla (v \circ M_{A,i}) \|_{0,S^{need}} 
\leq C \frac{p^2}{\lambda p \varepsilon} \|v \circ M_{A,i}\|_{0,S^{need}} \sim 
C \frac{p^2}{\lambda p \varepsilon} \|v \|_{0,S^{need}},  
$$
where we exploited that $v \circ M_{A,i}$ is a polynomial of degree $p$ and used
the inverse estimate \cite[Thm.~{3.91}]{schwab98}.
\end{proof}
We mention already at this point that we will quantify the contributions $z_1$ and $z_\varepsilon$
of this decomposition in Lemma~\ref{P02Dlemma} ahead. We close this section by pointing out that 
in our setting, one has very good control over the element maps:
There exist
$C > 0$ (depending solely on the asymptotic mesh $\Delta_A$) such that 
\begin{subequations}
\label{eq:control-element-maps}
\begin{eqnarray}
\label{eq:control-element-maps-1}
\|M^\prime_{A,i} \|_{\infty,\Sref} + 
\|(M^\prime_{A,i})^{-1} \|_{\infty,\Sref} 
&\leq & C, 
\quad i=1,\ldots,N^\prime, \\
\label{eq:control-element-maps-2}
\|(M^{reg}_{i})^\prime \|_{\infty,\Sref} + 
\|((M^{reg}_{i})^\prime)^{-1} \|_{\infty,\Sref}
&\leq & C, 
\quad i=1,\ldots,n, \\ 
\label{eq:control-element-maps-3}
\|(M^{need}_{i})^\prime \|_{\infty,\Sref} + 
\|((M^{need}_{i})^\prime)^{-1} \|_{\infty,\Sref}
&\leq & C \frac{1}{\lambda p \varepsilon},
\quad i=1,\ldots,n.
\end{eqnarray}
\end{subequations}

\subsection{Approximation properties of the Spectral Boundary Layer mesh}

By construction, the resulting mesh (in the case $\lambda p \varepsilon < 1/2$)
$$\Delta _{BL} = \Delta_{BL}(\lambda,p) 
=\left\{ \Omega _{1}^{need},...,\Omega_{n}^{need},
\Omega _{1}^{reg},...,\Omega _{n}^{reg},\Omega _{n+1},...,\Omega_{N}\right\} 
$$ 
is a regular admissible mesh in the sense of \cite{MelenkSchwab}. Therefore, 
\cite{MelenkSchwab} gives that the space 
$$
S_0(\lambda,p):= {\mathcal S}^p_0(\Delta_{BL}(\lambda,p))
$$
has the following approximation properties:  

\begin{proposition}[\cite{MelenkSchwab}]
\label{2Dapprox}
Let $u$ be the solution to (\ref{BuvFv2D}) and assume that (\ref{analytic_data_2D})
holds. Then there exist constants $\lambda_0$, $\lambda_1$, $C$, $\beta > 0$ 
independent of $\varepsilon \in (0,1]$ and $p \in \N$, such that 
the following is true: For every $p$ and every $\lambda \in (0,\lambda_0]$ with 
$\lambda p \ge \lambda_1$ 
there exists $\pi_{p} u \in \mathcal{S}_{0}^{{p}}(\Delta_{BL}(\lambda,p) )$ 
such that
\begin{equation*}
\left\Vert u-\pi_{p} u \right\Vert _{\infty, \Omega}+\varepsilon 
\left\Vert \nabla (u-\pi_{p} u) \right\Vert _{\infty,\Omega }\leq Cp^{2}\left( \ln
p+1\right) ^{2}e^{-\beta p\lambda }.
\end{equation*}
%
%
\end{proposition} 
We mention in passing that Proposition~\ref{2Dapprox} provides robust exponential convergence 
in the energy norm. However, as in the 1D case of Lemma~\ref{lemma:1D-approximation}, we can modify
the boundary layer part of the approximant of Proposition~\ref{2Dapprox}, so as to be able to 
approximate at a robust exponential rate in the balanced norm. This is achieved with the following 
2D analog of Lemma~\ref{lemma:1D-bdy-layer-approximation}.
\begin{lem} 
\label{lemma:2D-bdy-layer-approximation}
Let $v$ be defined on $S = [0,1]^2$, and let $v$ be analytic on $[0,d_0] \times [0,1]$ for some fixed 
$d_0 \in (0,1]$. Assume that for some $C_v$, $\gamma_v>0$ and $\varepsilon \in (0,1]$, the function
$v$ satisfies the following hypotheses: 
\begin{itemize}
\item[(R1)]
For every $\xi \in (0,d_0)$, the {\em stretched function}
$\widehat v_{\xi}: S \rightarrow \R$ given by $\widehat v_{\xi} (x,y):= u(x \xi,y)$, satisfies 
$$
\|D^\alpha \widehat v_{\xi} \|_{\infty,S} \leq C_v \gamma_v^{|\alpha|} \max\{|\alpha|+1,\xi/\varepsilon\}^{|\alpha|} 
\qquad \forall \alpha \in \N_0^2.
$$
\item[(R2)] 
The function $v$ satisfies 
$$
\sup_{y \in [0,1]} |\nabla^n v(x,y)| \leq C_v \varepsilon^{-n} e^{-x/\varepsilon} \qquad \forall x \in [0,1], 
\quad n \in \{0,1\}. 
$$
\end{itemize}
Then there are constants $C$, $\beta$, $\eta > 0$ (depending only on $\gamma_v$) such that under the assumption 
$$
\frac{\xi}{p \varepsilon} \leq \eta,
$$
the following is true for the mesh $\Delta_\xi = \{S_\xi^{need}, S_\xi^{reg}\}$ with 
$S_\xi^{need}:= [0,\xi]\times [0,1]$ and 
$S_\xi^{reg}:= [\xi,1]\times [0,1]$:  
There there is a piecewise polynomial  approximation $I_p v \in {\mathcal S}^{p}(\Delta_\xi)$ with
the following properties: 
\begin{enumerate}[(i)]
\item
On the  two edges $x = 0$ and $x = 1$ of $S$, the approximation $I_p v$ coincides with the Gau{\ss}-Lobatto
interpolant of $v$. On the edge $(0,\xi) \times \{0\}$, $I_p v$ is given by the Gau{\ss}-Lobatto interpolant
corrected by $(1-\sqrt{\varepsilon}) \frac{x}{\xi} v(\xi,0)$ (so that $(I_p v)(\xi,0) = \sqrt{\epsilon} v(\xi,0)$), 
and on the edge $(\xi,1) \times \{0\}$, $I_p v$ is the linear polynomial interpolating the values 
$\sqrt{\varepsilon} v(\xi,0)$ and $v(1,0)$ at the endpoints. 
$I_p v$ is defined analogously on the edges $(0,\xi) \times \{1\}$ and $(\xi,1) \times \{1\}$.
\item The approximation $I_p v$ satisfies 
\begin{align*}
& \|(v - I_p v)\|_{\infty,S_\xi^{need}} + \xi \|\partial_x (v - I_p v)\|_{\infty,S_\xi^{need}} 
+ \|\partial_y (v - I_p v)\|_{\infty,S_\xi^{need} }
 \\
&
\leq C C_v \left[ e^{-\beta p} + p^2 (1 + \ln p)^2  e^{-\xi/\varepsilon}\right], \\
& \|v - I_p v\|_{\infty,S_\xi^{reg}} \leq C C_v (1+\ln p)^2 e^{-\xi/\varepsilon}, \\ 
& \|v - I_p v\|_{0,S_\xi^{reg}} + \varepsilon \|\nabla (v - I_p v)\|_{0,S_\xi^{reg}} \leq 
C C_v p^2 (1+\ln p)^2 \sqrt{\varepsilon} e^{-\xi/\varepsilon}. 
\end{align*}
\end{enumerate}
\end{lem}
\begin{proof}
As in the corresponding 1D result (Lemma~\ref{lemma:1D-bdy-layer-approximation}), we construct $I_p v$
in two steps. In the first step, we study the approximation $v_1$ which is given by the piecewise 
Gau{\ss}-Lobatto interpolant. In the second step, we modify $v_1$ to obtain the additional factor 
$\sqrt{\varepsilon}$ for the error in the $L^2$-based norms 
on the large element $S_\xi^{reg}$.

{\em Step 1:}
For $\xi \in (0,d_0)$, let $v_1$ be the piecewise Gau{\ss}-Lobatto interpolant of $v$ on the
mesh $\Delta_\xi$. For simplicity, we assume $\xi \leq 1/2$. 
The error analysis for $v - v_1$ can be extracted from the proof of 
\cite[Thm.~{3.12}]{MelenkSchwab}; we highlight here the main arguments for completeness' sake. 
The one-dimensional Gau{\ss}-Lobatto interpolation operator $i_p:C([0,1]) \rightarrow \Pi_p$ has 
the stability property $\|i_p\|_{\infty,[0,1]} \leq C (1 + \ln p)$ by \cite{suendermann80}. Together
with a polynomial inverse estimate (Markov's inequality) we get on $S_\xi^{reg}$:
\begin{align*}
\|v_1\|_{\infty,S_\xi^{reg}} &\leq C (1 + \ln p)^2 \|v\|_{\infty,S_\xi^{reg}} 
\leq C C_v (1 + \ln p)^2 e^{-\xi/\varepsilon}, \\
\|\nabla v_1\|_{\infty,S_\xi^{reg}} &\leq C p^2 \|v_1\|_{\infty,S_\xi^{reg}} 
\leq C p^2 (1 + \ln p)^2 \|v\|_{\infty,S_\xi^{reg}} 
\leq C C_v p^2 (1 + \ln p)^2 e^{-\xi/\varepsilon}. 
\end{align*}
The error analysis for the Gau{\ss}-Lobatto interpolation on $S_\xi^{need}$
is achieved by (anisotropically) scaling $S_\xi^{need}$ to the reference element $S = [0,1]^2$. 
In order to make use of the regularity properties of the scaled function  $\widehat v$, we first observe that
for $n \in \N_0$
\begin{align*}
\max\{n+1,\xi/\varepsilon\}^n  &= 
\max\{(n+1)^n,\frac{1}{n!} (\xi/\varepsilon)^n  n!\}  
\leq \max\{(n+1)^n,n! e^{\xi/\varepsilon}\} 
\leq n! e^{\xi/\varepsilon} \frac{(n+1)^n}{n!} 
\\
&\leq C n! e^n e^{\xi/\varepsilon},
\end{align*} 
for some $C > 0$, where the last inequality follows from Stirling's formula. 
The tensor product Gau{\ss}-Lobatto interpolant $\widehat v_1$ of the stretched function $\widehat v_\xi$ 
satisfies on $S$ 
$$
\|\widehat v_\xi - \widehat v_1 \|_{\infty,S} + 
\|\nabla( \widehat v_\xi - \widehat v_1) \|_{\infty,S} 
\leq C C_v e^{\xi/\varepsilon} e^{-\beta p},  
$$
for some $C$, $\beta > 0$ that depend solely on $\gamma_v$. Returning to $S_\xi^{need}$, we get  for the 
Gau{\ss}-Lobatto interpolation error 
\begin{align*}
\|v - v_1 \|_{\infty,S_\xi^{need}} + 
\xi \|\partial_x ( v - v_1) \|_{\infty,S_\xi^{need}} + 
\|\partial_y ( v - v_1) \|_{\infty,S_\xi^{need}} 
\leq C C_v e^{\xi/\varepsilon} e^{-\beta p}.  
\end{align*}

{\em Step 2:} We define $I_p v$ as follows (thus correcting $v_1$): 
$$
I_p v(x,y):= 
\begin{cases}
v_1(x,y) - (1-\sqrt{\varepsilon}) v_1(\xi,y) \frac{x}{\xi}, & (x,y) \in S_\xi^{need} \\
\sqrt{\varepsilon} v_1(\xi,y) \frac{1-x}{1-\xi} + \frac{x-\xi}{1-\xi} v_1(1,y), & (x,y) \in S_\xi^{reg}.  
\end{cases}
$$
We note 
\begin{align*}
&\sup_{y \in [0,1]} |v_1(\xi,y)| \leq C C_v (1 + \ln p)^2 e^{-\xi/\varepsilon}, 
\qquad   
\sup_{y \in [0,1]} |\partial_y v_1(\xi,y)| \leq C C_v p^2 (1 + \ln p)^2 e^{-\xi/\varepsilon}, \\
&\sup_{y \in [0,1]} |v_1(1,y)| \leq C C_v (1 + \ln p)^2 e^{-1/\varepsilon}, 
\qquad   
\sup_{y \in [0,1]} |\partial_y v_1(1,y)| \leq C C_v p^2 (1 + \ln p)^2 e^{-1/\varepsilon}. 
\end{align*}
From this, we get on $S_\xi^{need}$ 
\begin{align*}
&\|v - I_p v\|_{\infty,S_\xi^{need}} + 
\xi \|\partial_x (v - I_p v)\|_{\infty,S_\xi^{need}} + 
\|\partial_y (v - I_p v)\|_{\infty,S_\xi^{need}} \\
&
\leq C C_v \left[ e^{\xi/\varepsilon} e^{-\beta p} + p^2 (1+\ln p)^2 e^{-\xi/\varepsilon} \right].
\end{align*} 
The hypothesis $\xi/\varepsilon \leq \eta p$ implies that 
$e^{\xi/\varepsilon} e^{-\beta p} \leq e^{(\eta - \beta)p}$, so that $\eta < \beta$ guarantees 
exponential convergence (in $p$). The claimed approximation properties on $S_\xi^{need}$ follow.

The approximations on $S_\xi^{reg}$ are achieved by the triangle inequality
$\|v - I_p v\| \leq \|v\| + \|I_p v\|$. The control of $I_p v$ is easily achieved by observing 
$$
\|I_p v\|_{\infty,S_\xi^{reg}} \leq C C_v (1 + \ln p)^2 \sqrt{\varepsilon} e^{-\xi/\varepsilon} \;, 
\; 
\|\nabla I_p v\|_{\infty,S_\xi^{reg}} \leq C C_v p^2 (1 + \ln p)^2 \sqrt{\varepsilon} e^{-\xi/\varepsilon}.
$$
Note that we suppressed the contributions arising from $v_1(1,\cdot)$ since our assumption $\xi \leq 1/2$ provides 
$e^{-1/\varepsilon} \leq C \sqrt{\varepsilon} e^{-\xi/\varepsilon}$ for some $C > 0$. 
\end{proof}
The improved treatment of the boundary layer contribution allows us to sharpen 
the approximation result of Proposition~\ref{2Dapprox} in the balanced norm: 
\begin{cor}
\label{cor:2Dapprox2}
Under the assumptions of Proposition~\ref{2Dapprox}, there exist constants 
$\lambda_0$, $\lambda_1$, $C$, $\beta > 0$ 
independent of $\varepsilon \in (0,1]$ and $p \in \N$, such that 
the following is true: For every $p$ and every $\lambda \in (0,\lambda_0]$ with 
$\lambda p \ge \lambda_1$,  
there exists $\widetilde \pi_{p} u \in \mathcal{S}_{0}^{{p}}(\Delta_{BL}(\lambda,p) )$ 
such that
\begin{equation*}
\|u- \widetilde \pi_{p} u\|_{\infty,\Omega} + 
\varepsilon^{1/2}\left\Vert \nabla (u-\widetilde \pi_{p} u) \right\Vert _{0,\Omega }\leq Cp^{2}\left( \ln
p+1\right) ^{2}e^{-\beta p\lambda }.
\end{equation*}
\end{cor}
\begin{proof}
In the case that the mesh $\Delta_{BL}(\lambda,p)$ consists of the asymptotic mesh
$\Delta_A$, we set $\widetilde \pi_p u = \pi_p u$ and the proof follows easily from 
Proposition~\ref{2Dapprox}, since 
$\varepsilon \ge 1/(2 \lambda p) \ge 1/(2 \lambda_1)$. Let, therefore, $\Delta_{BL}(\lambda,p)$ have needle
elements, i.e., the elements $\Omega_i$, $i=1,\ldots,n$ of the asymptotic mesh $\Delta_A$ are further
subdivided into $\Omega_i^{need}$ and $\Omega_i^{reg}$. Our starting point is the proof of 
Proposition~\ref{2Dapprox} in \cite{MelenkSchwab}. There, the approximation is obtained by a piecewise
Gau{\ss}-Lobatto interpolation of the function $u$, which is decomposed into a 
smooth (analytic) part $w$, a boundary layer part $u^{BL}$, and a remainder $r$: 
$$
u = w + u^{BL} + r. 
$$ 
The approximations of the smooth part $w$ and the remainder $r$ are taken to be those of 
\cite{MelenkSchwab}, i.e., the elementwise Gau{\ss}-Lobatto interpolants. The boundary layer part $u^{BL}$, however,
is not approximated by its elementwise Gau{\ss}-Lobatto interpolant but by the elementwise Gau{\ss}-Lobatto
interpolant on the elements $\Omega_i$ with $\overline{\Omega_i } \cap \partial\Omega = \emptyset$,  
with the aid of the operator $I_p$ of Lemma~\ref{lemma:2D-bdy-layer-approximation}. Inspection of the procedure
in \cite{MelenkSchwab} shows that the regularity hypotheses (R1), (R2) 
of Lemma~\ref{lemma:2D-bdy-layer-approximation} are satisfied and that the approximation result
holds if $\xi = \lambda p \varepsilon$ with $\lambda \leq \lambda_0$ and $\lambda_0$ sufficiently small. 
\end{proof}
\subsection{Robust exponential convergence in balanced norms}
The main result of the paper is the following robust exponential convergence in the 
balanced norm: 
\begin{thm} 
\label{thm:balanced-norm-2D}
There is a $\lambda _{0}>0$ depending only on the functions $b$, $f$ 
and the asymptotic mesh $\Delta_A$ 
such that for every $\lambda \in (0,\lambda _{0}]$, $\varepsilon \in (0,1]$, $p \in \N$, 
the $hp$-FEM space $\mathcal{S}_{0}^{p}(\Delta _{BL}(\lambda,p))$ leads to a finite element 
approximation $u_{FEM}\in \mathcal{S}_{0}^{p}(\Delta _{BL}(\lambda,p))$ satisfying 
\begin{equation*}
\sqrt{\varepsilon }\Vert \nabla (u-u_{FEM})\Vert _{0,\Omega }+
\Vert u-u_{FEM}\Vert _{0,\Omega }\leq Ce^{-\beta p};
\end{equation*}
the constants $C$, $\beta>0$ depend on the choice of $\lambda $ but are
independent of $\varepsilon $ and $p$.
\end{thm}%
The proof is deferred to the end of the section. 
As a corollary, we get exponential convergenence in the maximum norm.

\begin{cor} 
\label{cor:max-norm-estimate-2D}
Let $u$ be the solution of (\ref{BuvFv2D}) and let 
$u_{FEM}\in {\mathcal S}_{0}^{p}(\Delta _{BL}(\lambda,p))$ be its finite element approximation.
Then there exist constants $C$, $\sigma >0$ independent of $\varepsilon\in(0,1]$ and $p\in\N$ such that
\begin{equation*}
\left\Vert u-u_{FEM}\right\Vert _{\infty, \Omega}\leq C e^{-\sigma p}.
\end{equation*}
\end{cor}
\begin{proof} First we note that Corollary~\ref{cor:2Dapprox2} provides an approximation  
$\pi_p u \in {\mathcal S}_{0}^{p}(\Delta _{BL}(\lambda,p))$ with 
\[
\left\Vert u-\pi_p u \right\Vert _{\infty, \Omega}\leq C e^{-\beta \lambda p}.
\]
In view of the triangle inequality 
$\displaystyle 
\left\Vert u-u_{FEM}\right\Vert _{\infty, \Omega}\leq 
\left\Vert u-\pi_p u \right\Vert _{\infty, \Omega} + 
\left\Vert \pi_p u - u_{FEM}\right\Vert _{\infty, \Omega},
$
we may focus on the term $\left\Vert \pi_p u-u_{FEM}\right\Vert _{\infty, \Omega}$. 
It suffices to prove the result in the layer region, i.e., for the elements
$\Omega_i^{need}$, since outside $\Omega_{\lambda p \varepsilon}$ standard inverse estimates 
(bounding the $L^\infty$-norm of polynomials by their $L^2$-norm up to powers of $p$) yield 
the desired bound in view of 
(\ref{eq:control-element-maps-1}),
(\ref{eq:control-element-maps-2}). 

For a needle element $\Omega_i^{need}$ we introduce 
$\widetilde \pi_p u:= \pi_p u|_{\Omega_i^{need}} \circ M_{A,i}$ and 
$\widetilde u_{FEM}:= u_{FEM}|_{\Omega_i^{need}} \circ M_{A,i}$. The polynomial inverse estimate 
of \cite[Thm.~{4.76}]{schwab98} and an affine scaling argument (between $\Sref$ and $S^{need}$) yield 
\begin{align*}
\left\Vert \pi_p u - u_{FEM}\right\Vert _{\infty, \Omega_i^{need}}  & = 
\left\Vert \widetilde \pi_p u - \widetilde u_{FEM}\right\Vert _{\infty, S^{need}}
\leq C \frac{p^2}{\sqrt{\lambda p \varepsilon}} 
\left\Vert \widetilde \pi_p u - \widetilde u_{FEM}\right\Vert_{0,S^{need}}\\
& \sim  \frac{p^2}{\sqrt{\lambda p \varepsilon}} 
\left\Vert \pi_p u - u_{FEM}\right\Vert_{0,\Omega_i^{need}}, 
\end{align*}
where in the last step we used the assumptions on the element maps $M_{A,i}$. 
The triangle inequality then gives
\begin{equation}
\left\Vert \pi_p u - u_{FEM}\right\Vert _{\infty, \Omega_i^{need}} \leq C \frac{p^2}{\sqrt{\lambda p \varepsilon}} \left[ 
\left\Vert \pi_p u - u\right\Vert_{0,\Omega_i^{need}} 
+ \left\Vert u - u_{FEM} \right\Vert_{0,\Omega_i^{need}}  \right]. \label{maxNest1}
\end{equation}
For the first term in (\ref{maxNest1}) we obtain from the $L^\infty$-bound of Corollary~\ref{cor:2Dapprox2}
and the fact that $|\Omega_i^{need}| \sim \lambda p \varepsilon$, 
\begin{equation}
\label{maxNest2}
\left\Vert \pi_p u - u\right\Vert_{0,\Omega_i^{need}} \lesssim 
\sqrt{\lambda p \varepsilon} e^{-\beta p}. 
\end{equation}
For the second term in (\ref{maxNest1}) we exploit the fact that $u_{FEM} = 0 = \pi_p u$ on $\partial\Omega$ 
and a 1D Poincar\'e inequality. To that end, we note that for any function $\widetilde v \in H^1(S^{need})$ 
with $v = 0$ on the edge $\{(0,y)\,|\, 0 \leq y \leq 1\}$ 
of $S^{need} = \{(x,y)\,|\, 0 \leq x \leq \lambda p \varepsilon, 0 \leq y \leq 1\}$, 
we obtain from a 1D Poincar\'e inequality 
\begin{equation}
\label{eq:cor:max-norm-2D-10}
\|\widetilde  v\|_{0,S^{need}} \leq C \sqrt{\lambda p \varepsilon} \|\partial_x \widetilde v\|_{0,S^{need} }
\leq C \sqrt{\lambda p \varepsilon} \|\nabla  \widetilde  v\|_{0,S^{need}}. 
\end{equation}
Upon setting $\widetilde v:= (u - u_{FEM})|_{\Omega_i^{need}} \circ M_{A,i}$, we may use 
(\ref{eq:cor:max-norm-2D-10}) together with the properties of $M_{A,i}$ to get 
\begin{equation}
\label{eq:cor:max-norm-2D-20}
\|u - u_{FEM}\|_{0,\Omega_i^{need}} \sim 
\|\widetilde v\|_{0,S^{need}} \leq C \sqrt{\lambda p \varepsilon} \|\nabla \widetilde v\|_{0,S^{need}} 
\sim  \sqrt{\lambda p \varepsilon} \|\nabla (u - u_{FEM})\|_{0,\Omega_i^{need}}. 
\end{equation}
Combining (\ref{maxNest1}), (\ref{maxNest2}), (\ref{eq:cor:max-norm-2D-20}) gives the desired result.
\end{proof}

\subsection{Proof of Theorem~\ref{thm:balanced-norm-2D}}
The proof of Theorem~\ref{thm:balanced-norm-2D} parallels that of the 1D case in 
Section~\ref{sec:1D}. 
We begin by defining the bilinear form for the
reduced problem, 
\begin{equation}
{\mathcal{B}}_{0}(u,v)=\left\langle bu,v\right\rangle _{\Omega }.  \label{B0_2D}
\end{equation}
We also introduce the projection operator 
$\mathcal{P}_{0}:L^{2}(\Omega )\rightarrow \mathcal{S}_{0}^{p}(\Delta_{BL}(\lambda,p) )$ by the condition 
\begin{equation*}
{\mathcal{B}}_{0}\left( u-\mathcal{P}_{0}u,v\right) =0 \quad  \forall v\in 
\mathcal{S}_{0}^{p}(\Delta_{BL}(\lambda,p) ).
\end{equation*}
Then, by reasoning as in 
(\ref{eq:galerkin-orthogonality-1D}) with Galerkin orthogonalities, we get 
\begin{eqnarray*}
\left\Vert u_{FEM}-\mathcal{P}_{0}u\right\Vert _{E,\Omega }^{2}
&=&
\varepsilon ^{2}\left\langle \nabla \left( u-\mathcal{P}_{0}u\right) ,
\nabla \left( u_{FEM}-\mathcal{P}_{0}u\right) \right\rangle_{\Omega }.
\end{eqnarray*}
Hence 
\begin{equation*}
\varepsilon \left\Vert \nabla \left( u_{FEM}-\mathcal{P}_{0}u\right)
\right\Vert _{0,\Omega }\leq \left\Vert u_{FEM}-\mathcal{P}_{0}u\right\Vert_{E,\Omega }\leq 
\varepsilon \left\Vert \nabla \left( u-\mathcal{P}_{0}u\right) \right\Vert _{0,\Omega }.
\end{equation*}
The key step towards showing robust exponential convergence in balanced norms is therefore to show 
\begin{equation*}
\left\Vert \nabla \left( u-\mathcal{P}_{0}u\right) \right\Vert _{0,\Omega}
\leq C\varepsilon ^{-1/2}e^{-\sigma p},
\end{equation*}
for some $C$ and $\sigma >0$ independent of $\varepsilon$
and $p$. Completely analogous to the one-dimensional case, we are therefore led to 
studying the $H^{1}$-stability of the projection operator $\mathcal{P}_{0}$ on the 
\emph{Spectral Boundary Layer mesh} of Definition~\ref{SBL-2D}.
%

\begin{lem}[Strengthened Cauchy-Schwarz inequality in 2D] Let ${\mathcal{B}}_{0}$ be
given by (\ref{B0_2D}). Then, 
\begin{equation*}
\left\vert {\mathcal{B}}_{0}\left( u,v\right) \right\vert \leq C \min\{1,\sqrt{\lambda
p\varepsilon }p\}\left\Vert u\right\Vert _{0,\Omega }\left\Vert
v\right\Vert _{0,\Omega_{\lambda p \varepsilon} }\quad  \forall  u\in S_{1},\quad v\in S_{\varepsilon },
\end{equation*}
with $S_{1}$, $S_{\varepsilon }$ given by (\ref{S12D}) and (\ref{Sepsilon2D}), respectively. 
The constant $C > 0$ depends solely on $\|b\|_{\infty,\Omega}$, $\inf_{x \in \Omega} b(x) > 0$, and 
the element maps of the asymptotic mesh $\Delta_A$. 
\end{lem}
\begin{proof} We restrict our attention to the case $\lambda p \varepsilon <1/2$ as the ``$1$'' 
in the minimum is a simple consequence of the Cauchy-Schwarz inequality. 
With $u \in S_1$, $v \in S_{\varepsilon}$ there holds
${\mathcal{B}}_0 (u,v) = \iint_{\Omega_{\lambda p \varepsilon}} b u v$. 
Fix $\Omega_i^{need}$ and recall that it is obtained from an element $\Omega_i$ ($i \in \{1,\ldots,n\}$) 
by a splitting, i.e., $\overline{\Omega}_i  = \overline{\Omega^{need}_i} \cup \overline{\Omega^{reg}_i}$. 
The construction of $\Delta_{BL}(\lambda,p)$ implies that the pull-back $\pi_1:= u|_{\Omega_i} \circ M_{A,i}$ 
to $\Sref$
is a polynomial of degree $p$ (in each variable) whereas the pull-back 
$\pi_2:= v|_{\Omega_i} \circ M_{A,i}$ is a piecewise polynomial
of degree $p$ (in each variable) with $\operatorname*{supp} \pi_2 \subset S^{need}$. 
Upon setting $\widehat b:= b|_{\Omega_i^{need}} \circ M_{A,i}$, which is uniformly bounded on $S^{need}$, 
we calculate 
\begin{align*}
\iint_{\Omega_i} b uv \,dx\,dy & = 
\iint_{\Omega_i^{need}} b u v\,dx\,dy  = 
\iint_{S^{need}} \pi_1(x,y) \pi_2(x,y)\, \widehat b |\operatorname*{det} M_{A,i}^\prime| \,dx\, dy . 
\end{align*}
Since $|\operatorname*{det} M_{A,i}^\prime|$ is bounded 
uniformly (in $(x,y)$),  we obtain 
\[
\left\vert \iint_{\Omega_i^{need}} b u v \right\vert  \leq C \iint_{S^{need}} |\pi_1(x,y)| |\pi_2(x,y)| dx dy = C
\int_{0}^{1} \int_{0}^{\lambda p \varepsilon} |\pi_1(x,y)| |\pi_2(x,y)| dx dy.
\]
Now, fix $y \in [0,1]$ and consider
\[
\int_{0}^{\lambda p \varepsilon} |\pi_1(x,y)| |\pi_2(x,y)| dx  \leq 
C p \sqrt{\lambda p \varepsilon} \left[ \int_{0}^{1} |\pi_1(x,y)|^2 dx \right]^{1/2} 
\left[ \int_{0}^{\lambda p \varepsilon} |\pi_2(x,y)|^2 dx \right]^{1/2}
\]
by Lemma \ref{SCS}. Integrating in $y$ from 0 to 1, gives
\[
\int_0^{1}\int_{0}^{\lambda p \varepsilon} |\pi_1(x,y)| |\pi_2(x,y)| dx dy \leq 
C p \sqrt{\lambda p \varepsilon} \int_0^1 \left[ \int_{0}^{1} |\pi_1(x,y)|^2 dx \right]^{1/2} 
\left[ \int_{0}^{\lambda p \varepsilon} |\pi_2(x,y)|^2 dx \right]^{1/2} dy.
\]
Using once more the Cauchy-Schwarz inequality, we arrive at 
\[
\iint_{S^{need}} |\pi_1(x,y)| |\pi_2(x,y)| dx dy \leq 
C p \sqrt{\lambda p \varepsilon} \Vert \pi_1 \Vert_{0,\Sref} \Vert \pi_2 \Vert_{0,S^{need}}.
\]
The assumptions on the element map $M_{A,i}$ allows us to infer 
$\Vert \pi_1 \Vert_{0,\Sref} \Vert \pi_2 \Vert_{0,S^{need}} \sim 
\Vert u \Vert_{0,\Omega_i} \Vert v \Vert_{0,\Omega_i^{need}}$, which concludes the proof.  
\end{proof}

\begin{lem} [Stability of ${\mathcal P}_0$]
\label{P02Dlemma} There exist constants $C$, $c>0$
depending solely on $\|b\|_{\infty,\Omega}$, $\inf_{x \in \Omega} b(x) > 0$, 
and the element maps of the asymptotic mesh $\Delta_A$ such that the following 
is true
under the assumption
\begin{equation}
\sqrt{\lambda p\varepsilon} p \leq c:
\label{eq:discrete-scale-resolution-2D}
\end{equation}
For each $z \in L^2(\Omega)$, 
the (unique) decomposition 
\begin{equation*}
{\cal{P}} _{0}z=z_{1}+z_{\varepsilon }
\end{equation*}
into the components $z_1 \in S_{1}$ and $z_{\varepsilon} \in S_{\varepsilon }$ satisfies 
\begin{eqnarray}
\Vert z_{1}\Vert _{0,\Omega} &\leq &C\Vert z\Vert _{0,\Omega},  \label{z12D} \\
\Vert z_{\varepsilon }\Vert _{0,\Omega} &\leq & C\{\Vert z\Vert _{0,\Omega_{\lambda p \varepsilon
}}+\sqrt{\lambda p\varepsilon }p \Vert z\Vert _{0,\Omega}\}.  \label{zepsilon2D}
\end{eqnarray}
Furthermore, 
\begin{eqnarray}
\Vert \nabla z_{1}\Vert _{0,\Omega} &\leq &C p^2 \Vert z\Vert _{0,\Omega},  \label{z12D-H1} \\
\Vert \nabla z_{\varepsilon }\Vert _{0,\Omega} &\leq & 
C\frac{p^2}{\lambda p \varepsilon} \left \{\Vert z\Vert _{0,\Omega_{\lambda p \varepsilon }}
+\sqrt{\lambda p\varepsilon }p \Vert z\Vert _{0,\Omega}\right\}.  \label{zepsilon2D-H1}
\end{eqnarray}
\end{lem}
\begin{proof}
The proof parallels that of Lemma~\ref{lemma:almost-orthogonal}. With 
Lemma~\ref{lemma:2D-direct-sum} we can write ${\mathcal P}_0 z = z_1 + z_\varepsilon$. 
We define the auxiliary function $\psi_\varepsilon$ on $\Sref$ by 
\begin{equation*}
\psi _{\varepsilon }(x,y ):=\begin{cases}
\left( 1-\frac{ 2x }{\lambda p\varepsilon }\right)^{p} & \text{ if } (x,y) \in S^{need} \\ 
0 & \text{  otherwise.}
\end{cases}
\end{equation*}
Then $\text{supp }\psi _{\varepsilon }\subset S^{need},\psi _{\varepsilon }(0,y)=1$ and 
$\left\Vert \psi_{\varepsilon }\right\Vert _{0,\Sref} = 
\left\Vert \psi_{\varepsilon }\right\Vert _{0,S^{need}}\sim p^{-1/2}\sqrt{\lambda p\varepsilon }$.
We define the function $\widetilde z_\varepsilon \in S_\varepsilon$ on the needle elements $\Omega_i^{need}$ 
by prescribing its pull-back to $S^{need}$:  
\begin{equation*}
(\widetilde{z}_{\varepsilon }|_{\Omega_i^{need}} \circ M_{A,i}) (x,y) :=
(z_{\varepsilon }|_{\Omega_i^{need}} \circ M_{A,i})(x,y) 
+\psi _{\varepsilon}(x, y) (z_{1}|_{\Omega_i} \circ M_{A,i})(0,y), 
\qquad (x,y) \in S^{need}; 
\end{equation*}
here, $\Omega_i$ and $\Omega_i^{need}$ are related to each other by 
$\Omega_i = \Omega_i^{need} \cup \Omega_i^{reg}$.  
It is an effect of the assumptions on the asymptotic mesh $\Delta_A$ that the elementwise defined
function $\widetilde z_\varepsilon$ is in fact in $H^1(\Omega)$ and therefore indeed 
$z_\varepsilon \in S_\varepsilon$. By construction, 
$\widetilde z_\varepsilon|_{\partial\Omega} = (z_1 + z_\varepsilon)|_{\partial\Omega} = 
({\mathcal P}_0 z)|_{\partial\Omega} = 0$ so that 
$\widetilde z_\varepsilon \in S_\varepsilon \cap S_0(\lambda,p)$. 
Noting the product structure of $(z_\varepsilon - \widetilde z_\varepsilon)|_{\Omega_i^{need}} \circ M_{A,i}$ 
on $S^{need}$ and the above estimate on $\|\psi_\varepsilon\|_{0,S^{need}}$, we get for  
$\widetilde z_\varepsilon$ with the inverse estimate (\ref{inv2Daa}),
\begin{equation*}
\left\Vert \widetilde{z}_{\varepsilon }\right\Vert _{0,\Omega }=
\left\Vert \widetilde{z}_{\varepsilon }\right\Vert _{0,\Omega _{\lambda p\varepsilon}}\leq 
C\left\{ \left\Vert z_{\varepsilon }\right\Vert _{0,\Omega _{\lambda \varepsilon }}+
p^{1/2}\sqrt{\lambda p\varepsilon }\left\Vert z_{1}\right\Vert _{0,\Omega }\right\} .
\end{equation*}
We also have in view of ${\mathcal P}_0 z = z_1 + z_\varepsilon$ 
\begin{eqnarray}
B_{0}(z_{1},v_{1})+B_{0}(z_{\varepsilon },v_{1}) &=&B_{0}(\mathcal{P}_{0}z,v_{1})
\quad \forall  v_{1}\in S_{1},  \label{1_2D} \\
B_{0}(z_{1},v_{\varepsilon })+B_{0}(z_{\varepsilon },v_{\varepsilon })
&=&B_{0}(\mathcal{P}_{0}z,v_{\varepsilon })=B_{0}(z,v_{\varepsilon })
\quad 
\forall  v_{\varepsilon }\in {S}_{\varepsilon }\cap 
\mathcal{S}_{0}^{p}\left(\Delta _{BL}(\lambda,p)\right) ,  \label{2_2D}
\end{eqnarray}
where in (\ref{2_2D}) we used the fact that $\mathcal{P}_{0}$ is the ${\mathcal B}_0$--projection 
onto $\mathcal{S}_{0}^{p}\left( \Delta _{BL}(\lambda,p)\right)$. Taking $v_{1}=z_{1}$
in (\ref{1_2D}) and 
$v_{\varepsilon }=\widetilde{z}_{\varepsilon }\in S_{\varepsilon }\cap 
\mathcal{S}_{0}^{p}\left( \Delta _{BL}(\lambda,p)\right)$ in (\ref{2_2D}) yields, together with 
the Strengthened Cauchy Schwarz inequality of Lemma~\ref{lemma:almost-orthogonal}, just like in the 1D case, 
\iftechreport 
 \begin{eqnarray*}
 \Vert z_{1}\Vert _{0,\Omega }^{2} &\leq & C \left[
 \Vert \mathcal{P}_{0}z\Vert_{0,\Omega }\Vert z_{1}\Vert _{0,\Omega }+
 \sqrt{\lambda p\varepsilon} p \Vert z_{\varepsilon }\Vert_{0,\Omega }\Vert z_{1}\Vert _{0,\Omega} \right] \\
 \Vert z_{\varepsilon }\Vert _{0,\Omega }^{2} &\leq &
 C \left[ \Vert z\Vert _{0,\Omega}\Vert \widetilde{z}_{\varepsilon }\Vert _{0,\Omega}+
 \sqrt{\lambda p\varepsilon}p \Vert \widetilde{z}_{\varepsilon}\Vert _{0,\Omega}
 \Vert z_{1}\Vert _{0,\Omega}+\Vert z_{\varepsilon}\Vert _{0,\Omega}\Vert z_{1}\Vert_{0,\Omega}
 \sqrt{\lambda p\varepsilon }p^{1/2} \right], \\
 &\leq & C \left[ \Vert z_{\varepsilon }\Vert _{0,\Omega_{\lambda p \varepsilon}}
 \{ \Vert z\Vert_{0,\Omega _{\lambda p\varepsilon }}+
 \sqrt{\lambda p\varepsilon }p \Vert z_{1}\Vert _{0,\Omega }+
 \sqrt{\lambda p\varepsilon }p  \Vert z_{1}\Vert_{0,\Omega }\} \right. \\
 &&\left.\qquad \mbox{}+\{ \Vert z\Vert _{0,\Omega _{\lambda p\varepsilon} }+
 \sqrt{\lambda p\varepsilon}p \Vert z_{1}\Vert _{0,\Omega} \} 
 \sqrt{\lambda p\varepsilon}p^{1/2}\Vert z_{1}\Vert _{0,\Omega } \right].
 \end{eqnarray*}
 Estimating $\sqrt{\lambda p\varepsilon }p^{1/2}\leq \sqrt{\lambda p\varepsilon}p$ 
 and using an appropriate Young inequality we get 
 \begin{subequations}
 \begin{eqnarray}
 \Vert z_{1}\Vert _{0,\Omega } &\leq& C \left[ \Vert \mathcal{P}_{0}z\Vert _{0,\Omega}
 +\sqrt{\lambda p\varepsilon }p \Vert z_{\varepsilon }\Vert _{0,\Omega } \right],
 \label{lemma:stable-discrete-decomposition-scalar-600_2D} \\
 \Vert z_{\varepsilon }\Vert _{0,\Omega } &\leq &C\left[ 
 \Vert z \Vert_{0,\Omega _{\lambda p\varepsilon }}+
 \sqrt{\lambda p\varepsilon }p \Vert z_{1}\Vert _{0,\Omega }\right] .
 \label{lemma:stable-discrete-decomposition-scalar-600-b_2D}
 \end{eqnarray}
 \end{subequations}
 Inserting (\ref{lemma:stable-discrete-decomposition-scalar-600-b_2D}) in 
 (\ref{lemma:stable-discrete-decomposition-scalar-600_2D}), assuming that 
 $\sqrt{\lambda p\varepsilon }p $ is sufficiently small and using the
 stability $\Vert \mathcal{P}_{0}z\Vert _{0,\Omega }\leq C \Vert z\Vert_{0,\Omega }$ 
 gives $\Vert z_{1}\Vert _{0,\Omega }\leq C\Vert z\Vert_{0,\Omega }$. 
 Inserting this bound in (\ref{lemma:stable-discrete-decomposition-scalar-600-b_2D}) 
 concludes the proof of 
 (\ref{z12D}, (\ref{zepsilon2D}). 
 The final estimates
 (\ref{z12D-H1}), (\ref{zepsilon2D-H1}) follow from 
 (\ref{z12D}, (\ref{zepsilon2D}) with the aid of the inverse estimates 
 (\ref{inv2Da}), (\ref{inv2Db}) of Lemma~\ref{lemma:2D-direct-sum}. 
\else 
the bounds 
(\ref{z12D}), (\ref{zepsilon2D}). 
 The final estimates 
 (\ref{z12D-H1}), (\ref{zepsilon2D-H1}) follow from 
 (\ref{z12D}, (\ref{zepsilon2D}) with the aid of the inverse estimates 
 (\ref{inv2Da}), (\ref{inv2Db}) of Lemma~\ref{lemma:2D-direct-sum}. 
\fi 
\end{proof}

We are now in the position to prove the following

\begin{lem}
\label{lemma:estimate-Pi0u-scalar2D}
Assume (\ref{analytic_data_2D}) and let $u$ be the solution of (\ref{BuvFv2D}). 
Let $\lambda_0 > 0$ be given by Corollary~\ref{cor:2Dapprox2}. Assume that 
$\lambda \leq \lambda_0$ and that $\lambda$, $p$, $\varepsilon$ satisfy 
(\ref{eq:discrete-scale-resolution-2D}). Then, for constants 
$C$, $\beta > 0$ independent of $\varepsilon\in(0,1]$ and $p\in\N$ (but depending on 
$\lambda$) 
\begin{equation}
\label{eq:key-estimate-2D}
\| \nabla (u - {\cal{P}}_0 u) \|_{0,\Omega} \leq C \varepsilon^{-1/2} e^{-\beta p}.
\end{equation}
\end{lem}
\begin{proof} 
By Corollary~\ref{cor:2Dapprox2} we can find an approximation 
$\pi _{p}u\in \mathcal{S}_{0}^{p}(\Delta_{BL}(\lambda ,p))$ with 
$(u-\pi_{p}u)|_{\partial \Omega }=0$ such that 
\begin{equation*}
\sqrt{\varepsilon }\left\Vert \nabla (u-\pi_{p}u)\right\Vert _{0,\Omega }\leq 
C p^{2}\left( \ln p+1\right)^{2}e^{-\beta\lambda p}.
\end{equation*}
Since $\mathcal{P}_{0}(u-\pi_{p}u)\in \mathcal{S}^p_0(\Delta_{BL}(\lambda,p))$, we decompose 
${\mathcal{P}}_{0}(u-\pi_{p}u)=z_{1}+z_{\varepsilon }$ and use 
(\ref{z12D-H1}), (\ref{zepsilon2D-H1}), 
%
\begin{eqnarray}
\vert z_{1}\vert _{1,\Omega} &\lesssim &
p^{2}\Vert u-\pi_{p}u \Vert _{0,\Omega}\lesssim Ce^{-bp}, \\
\label{eq:2D-foo-10}
\vert z_{\varepsilon }\vert _{1,\Omega} &\lesssim &
\frac{p^{2}}{\lambda p\varepsilon }\left[ 
\Vert u-\pi_{p}u \Vert_{0,\Omega_{\lambda p \varepsilon }}+
\sqrt{\lambda p\varepsilon }p\Vert u-\pi_{p}u \Vert_{0,\Omega}\right] .
\end{eqnarray}
Let us treat the term $\Vert u-\pi_{p}u\Vert _{0,\Omega_{\lambda p \varepsilon}}$ above.  
Recall that $\Omega_{\lambda p \varepsilon} = \cup_{i=1}^n \Omega_i^{need}$; from  
(\ref{maxNest2}) we therefore get 
$\displaystyle 
\Vert u-\pi_{p}u\Vert _{0,\Omega_{\lambda p \varepsilon}} \lesssim \sqrt{\lambda p \varepsilon} e^{-\beta p}. 
$
Furthermore, from Corollary~\ref{cor:2Dapprox2} we readily have 
$\Vert u-\pi_{p}u\Vert _{0,\Omega} \lesssim e^{-\beta p}$. Inserting these two estimates 
into 
(\ref{eq:2D-foo-10}) produces 
\begin{equation*}
\vert z_{\varepsilon }\vert _{1,\Omega}\lesssim \frac{p^{2}}{\lambda p \varepsilon} 
\sqrt{\lambda p \varepsilon} e^{-\beta p} + 
\sqrt{\lambda p\varepsilon  p} e^{-\beta p}  \lesssim 
\varepsilon ^{-1/2}e^{-\beta p},
\end{equation*}
where the constant $\beta > 0$ is suitably adjusted in each estimate. 
The result follows. \end{proof}

\begin{numberedproof}{Theorem~\ref{thm:balanced-norm-2D}}  Again, we focus only on the control of 
$\sqrt{\varepsilon }\Vert \nabla (u-u_{FEM})\Vert _{0,\Omega }$. We distinguish two cases:

\emph{Case 1:} Assume that (\ref{eq:discrete-scale-resolution-2D}) is satisfied. Then 
(\ref{eq:key-estimate-2D}) and Lemma~\ref{lemma:estimate-Pi0u-scalar} yield the result.

\emph{Case 2:} Assume (\ref{eq:discrete-scale-resolution-2D}) is not satisfied. 
Then 
$\varepsilon \geq c^{2}p^{-3}\lambda ^{-1}$ so that 
\begin{equation*}
\sqrt{\varepsilon }\Vert \nabla (u-u_{N})\Vert _{0,\Omega }\leq 
\varepsilon^{-1/2}\Vert u-u_{N}\Vert _{E,\Omega }\leq 
\frac{1}{c}\sqrt{\lambda} p^{3/2}\Vert u-u_{N}\Vert _{E,\Omega }\lesssim e^{-bp}.
\end{equation*}
\end{numberedproof}

\subsection{Numerical example}

We close with a numerical example in two dimensions: We consider the problem 
\begin{eqnarray*}
-\varepsilon ^{2}\Delta u+u &=&1 \quad \text{ in }\Omega:=\left\{(x,y)\,|\, 0 \leq \left(\frac{x}{2}\right)^2 + y^2 < 1
	\right\} \subset \mathbb{R}^2,
\\
u &=&0 \quad \text{ on }\partial \Omega ,
\end{eqnarray*}
We approximate the solution to 
this problem on the mesh shown in Figure~\ref{fig:2D-mesh}  below, using polynomials of degree $1, ..., 7$.
\begin{figure}[ht]
\par
\begin{center}
\includegraphics[width=0.525\textwidth]{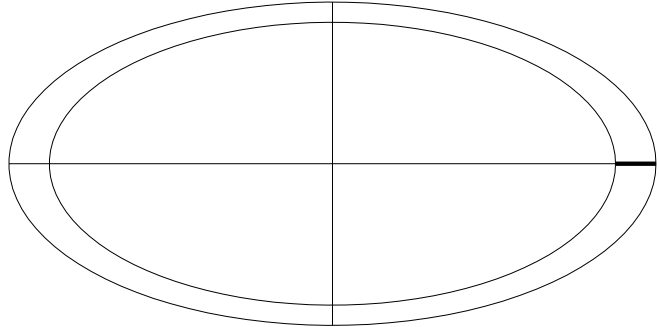}
\end{center}
\caption{\label{fig:2D-mesh} Mesh used for the two-dimensional example.}
\end{figure}
In Figure~\ref{fig:2D-results} we present the error 
\begin{equation*}
\max_{1\leq i \leq M }\left\vert u(r_i) - u_{FEM}(r_i) \right\vert, 
\qquad M:= 20, 
\end{equation*}
versus the polynomial degree $p$, in a semi-log scale. The $M$ points $r_i$
were uniformly distributed first on the mesh line connecting the points $(8 \varepsilon,0), (1,0)$,
as highlighted in Figure~\ref{fig:2D-mesh}, and second on the generic line, of width approximately $8\varepsilon$,
within the layer starting from the boundary point $(\sqrt{2},\sqrt{2}/2)$ at a $-45$ degree angle.
Figure~\ref{fig:2D-results} clearly shows the robust exponential convergence in the $L^\infty(\Omega)$-norm 
of the $hp$-FEM on the \emph{Spectral Boundary Layer mesh}.  

\begin{figure}[ht]
\par
\begin{center}
\includegraphics[width=0.45\textwidth]{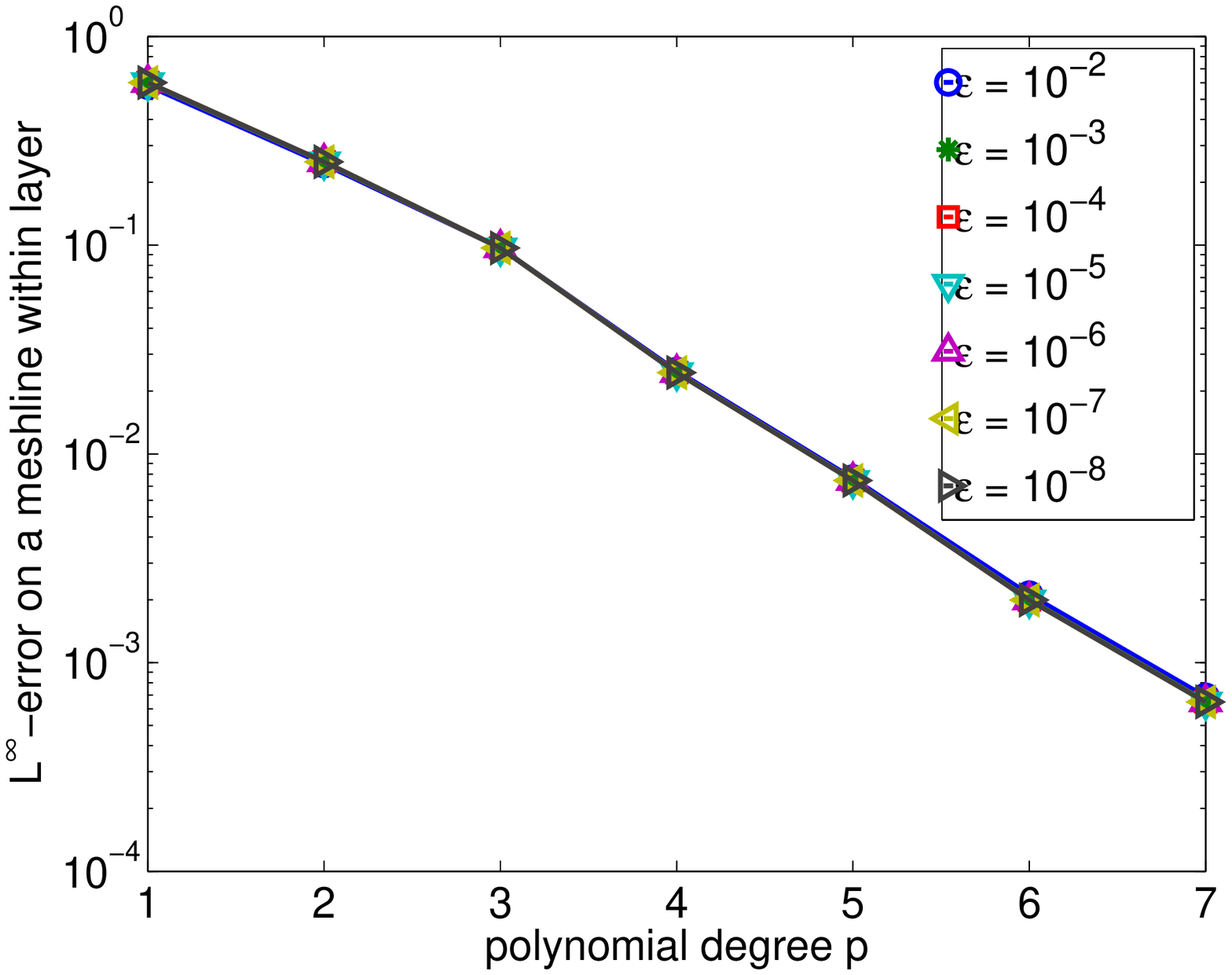}
\includegraphics[width=0.45\textwidth]{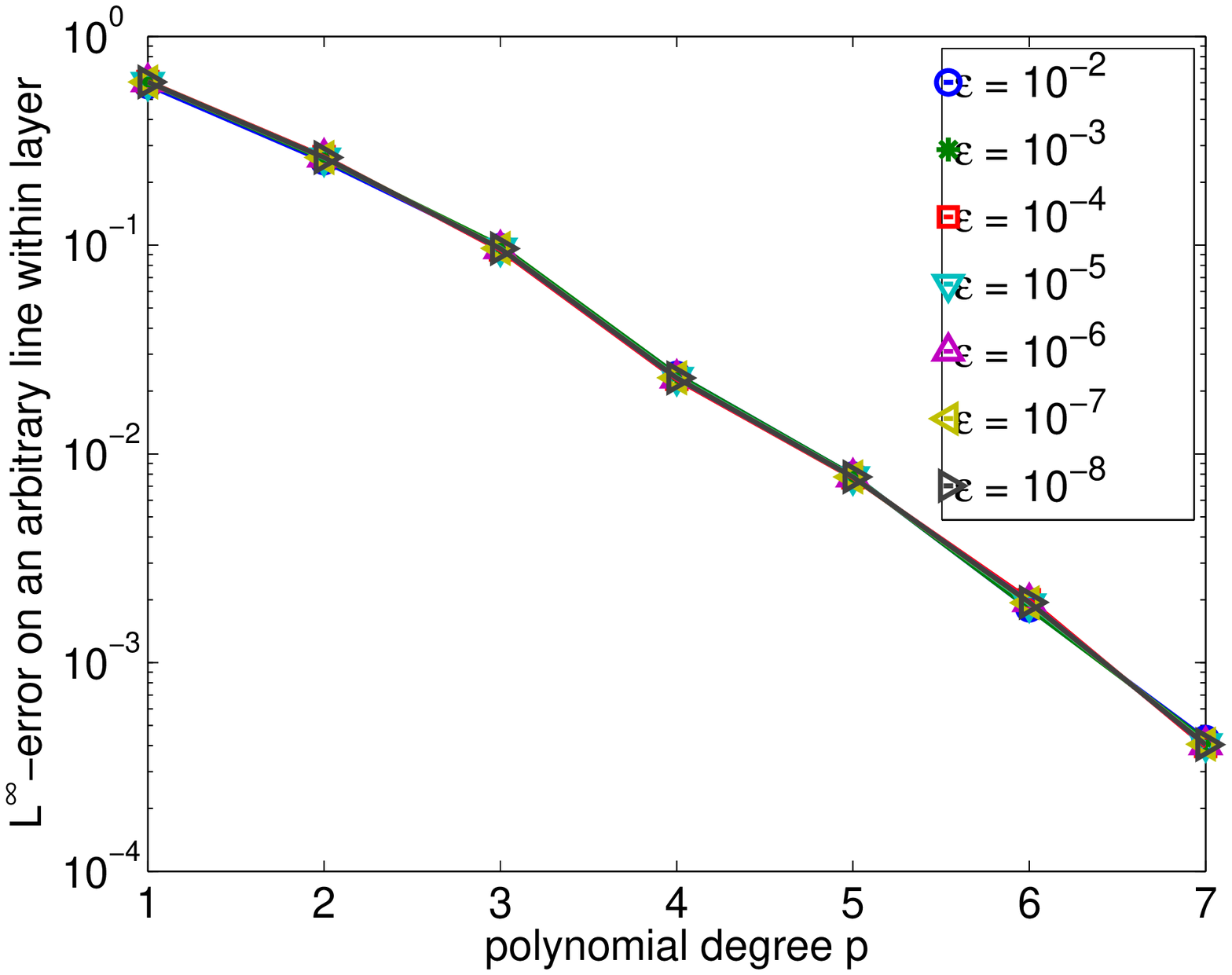}
\end{center}
\caption{\label{fig:2D-results} Maximum norm convergence of the $hp$-FEM. Left: on a meshline within the layer. Right: on a generic line within the layer.}
\end{figure}


\begin{thebibliography}{99}
\bibitem{B} N. S. Bakhvalov, \emph{Towards optimization of methods for
solving boundary value problems in the presence of boundary layers}, (in
Russian), Zh. Vychisl. Mat. \ Mat. Fiz. Vol. 9, pp. 841--859 (1969).

\bibitem{lin-stynes11} R. Lin and M. Stynes, \emph{A balanced finite element
method for singularly perturbed reaction-diffusion problems}, SIAM J. Numer.
Anal., Vol. 50, no.5, pp. 2729--2743 (2012).

\bibitem{melenk97} J. M. Melenk, \emph{On the robust exponential convergence
of hp finite element methods for problems with boundary layers}, IMA J. Num.
Anal., Vol. 17, pp. 577 -- 601 (1997).

\bibitem{mB} J. M. Melenk, \emph{hp-Finite Element Methods for Singular
Perturbations}, Vol. 1796 of Springer Lecture Notes in Mathematics, Springer
Verlag, 2002.

\bibitem{melenk-xenophontos-oberbroeckling13a} M. J. Melenk, C. Xenophontos
and L. Oberbroeckling, \emph{Robust exponential convergence of hp-FEM for
singularly perturbed systems of reaction-diffusion equations with multiple
scales}, IMA J. Num. Anal., Vol. 33., No 2, pp. 609--628, (2013).

\bibitem{melenk-xenophontos-oberbroeckling13b} M. J. Melenk, C. Xenophontos
and L. Oberbroeckling, \emph{Analytic regularity for a singularly perturbed
system of reaction-diffusion equations with multiple scales}, Advances in
Computational Mathematics, Vol. 39, pp. 367--394 (2013).

\bibitem{MelenkSchwab} J. M. Melenk and C. Schwab, \emph{hp FEM for reaction
diffusion equations I: Robust exponential convergence}, SIAM J. Num. Anal.,
Vol. 35, pp 1520--1557 (1998).

\bibitem{mos} J. J. Miller, E. O'Riordan, G. I. Shishkin, \emph{Fitted
Numerical Methods For Singular Perturbation Problems}, World Scientific,
1996.

\bibitem{roos-franz11} H. G. Roos and S. Franz, \emph{Error estimation in a
balanced norm for a convection-diffusion problems with two different
boundary layers}, in press in \emph{Calcolo}.

\bibitem{roos-schopf11} H. G. Roos and M. Schopf, \emph{Convergence and
stability in balanced norms of finite element methods on Shishkin meshes for
reaction-diffusion problems}, in press in \emph{ZAMM}.

\bibitem{rst} H. G. Roos, M. Stynes and L. Tobiska, \emph{Robust numerical
methods for singularly perturbed differential equations}, Volume 24 of \emph{%
Springer Series in Computational Mathematics}, Springer-Verlag, Berlin, 2008.

\bibitem{schwab98} C. Schwab, \emph{p/hp Finite Element Methods}, Oxford
University Press, 1998.

\bibitem{schwab-suri96} C. Schwab and M. Suri, \emph{The p and hp versions
of the finite element method for problems with boundary layers}, Math. Comp. 
\textbf{65},pp. 1403--1429 (1996).

\bibitem{Shishkin2} G. I. Shishkin, \emph{Grid approximation of singularly
perturbed boundary value problems with a regular boundary layer}, Sov. J.
Numer. Anal. Math. Model. Vol. 4, pp. 397--417 (1989).

\bibitem{suendermann80}
B. S\"undermann, \emph{{L}ebesgue constants in {L}agrangian interpolation at the {F}ekete
  points}, Ergebnisberichte der Lehrst\"uhle Mathematik III und VIII (Angewandte
  Mathematik)~44, {U}niversit\"at {D}ortmund, 1980.


\end{thebibliography}

\iftechreport
\appendix
\include{appendix}
\fi
\end{document}